\documentclass[11pt]{article}
\usepackage{amsmath,amsthm,amsfonts,amssymb} 
\usepackage[all]{xy}
\usepackage{fullpage}

\theoremstyle{plain}
\newtheorem{theorem}{Theorem}[section]
\newtheorem{lemma}[theorem]{Lemma}
\newtheorem{corollary}[theorem]{Corollary}
\newtheorem{proposition}[theorem]{Proposition}

\theoremstyle{definition}
\newtheorem{definition}[theorem]{Definition}

\theoremstyle{remark}
\newtheorem{remark}[theorem]{Remark}

\numberwithin{equation}{section}

\newcommand{\cA}{\mathcal{A}}
\newcommand{\cB}{\mathcal{B}}
\newcommand{\cC}{\mathcal{C}}

\newcommand{\cF}{\mathcal{F}}

\newcommand{\cH}{\mathcal{H}}
\newcommand{\cI}{\mathcal{I}}
\newcommand{\cJ}{\mathcal{J}}

\newcommand{\cU}{\mathcal{U}}

\newcommand{\RR}{\mathbb{R}}

\newcommand{\CC}{\mathbb{C}}

\newcommand{\unit}{\mathbf{1}}
\newcommand{\id}{\mathrm{id}}

\newcommand{\iotabar}{\overline{\iota}}
\newcommand{\eps}{\varepsilon}

\newcommand{\GLoc}{G\!-\!\mathrm{Loc}}
\newcommand{\GpLoc}{G'\!-\!\mathrm{Loc}}

\newcommand{\Mob}{\mathsf{M\ddot{o}b}}
\newcommand{\PSU}{\mathsf{PSU}}

\newcommand{\tilbeta}{\widetilde{\beta}}
\newcommand{\tilgamma}{\widetilde{\gamma}}

\newcommand{\tillambda}{\widetilde{\lambda}}

\newcommand{\malpha}[1]{\alpha^-_{#1}}
\newcommand{\pmalpha}[1]{\alpha^\pm_{#1}}

\newcommand{\gpalpha}[2]{\alpha^{#1;+}_{#2}}
\newcommand{\gmalpha}[1]{\alpha^{-}_{#1}}

\DeclareMathOperator{\Ad}{Ad}
\DeclareMathOperator{\End}{End}
\DeclareMathOperator{\Aut}{Aut}
\DeclareMathOperator{\Hom}{Hom}

\DeclareMathOperator{\Rep}{Rep}
\DeclareMathOperator{\Diff}{Diff}

\title{\bf On Induction for Twisted Representations of Conformal Nets}

\author{
	{\sc Ryo Nojima} \\
	{\small Graduate School of Mathematical Sciences}\\
	{\small The University of Tokyo, Komaba, Tokyo, 153-8914, Japan}\\
    {\small E-mail: {\tt nojima@ms.u-tokyo.ac.jp}}
}
\date{\empty}

\begin{document}

\maketitle

\begin{abstract}
For a given finite index inclusion of conformal nets $\cB\subset \cA$ and a group $G < \Aut(\cA, \cB)$, we consider the induction and the restriction procedures for $G$-twisted representations.
We define two induction procedures for $G$-twisted representations, which generalize the $\alpha^{\pm}$-induction for DHR endomorphisms.
One is defined with the opposite braiding on the category of $G$-twisted representations as in $\alpha^-$-induction.
The other is also defined with the braiding, but additionally with the $G$-equivariant structure on the Q-system associated with $\cB\subset \cA$ and the action of $G$.
We derive some properties and formulas for these induced endomorphisms in a similar way to the case of ordinary $\alpha$-induction.
We also show the version of $\alpha\sigma$-reciprocity formula for our setting.
\end{abstract}

\section{Introduction} \label{section:introduction}
In the Haag-Kastler framework of quantum field theory, a chiral components of 2D conformal field theory is described by a conformal net on the unit circle $S^1$.
A conformal net $\cA$ is defined to be a map $I\mapsto \cA(I)$ from the set of open intervals of $S^1$ to that of von Neumann algebras. These von Neumann algebras are considered as algebras of observables and required to satisfy certain axioms.
We have a natural notion of representations of $\cA$ and the representation theory plays an important role in the study of conformal nets.
By the Doplicher-Haag-Roberts theory~\cite{DHRI71,DHRII74}, it turns out that every representation is equivalent to a localized transportable endomorphism (called a DHR endomorphism) and $\Rep(\cA)$ has a structure of a braided C*-tensor category.

For a given conformal net, we can consider its extensions and subnets.
Let $\cB$ a conformal net and $\cA$ an extension. This gives us a net of subfactors $\{\cB(I)\subset \cA(I)\}$. In the article~\cite{LR95}, general theory of nets of subfactors has been developed. By applying their results, the extension $\cA$ is completely characterized by a commutative Q-system (or standard C*-Frobenius algebra objects) $\Theta=(\theta, w, x)$ in $\Rep(\cB)$.
We also have induction and restriction procedures for representations of $\cB$ and $\cA$.
The induction procedure is called the $\alpha$-induction and the restriction procedure is called the $\sigma$-restriction, respectively.
The notions of $\alpha$-induction and $\sigma$-restriction were first introduced in~\cite{LR95}, and studied with examples  in~\cite{Xu98}, and then further developed in~\cite{BE1,BE2,BE3}.
For a DHR endomorpshim $\lambda$ of $\cB$, $\alpha^{\pm}_{\lambda}$ is given as an extension of $\lambda$ to $\cA$. The endomorphism $\alpha^{\pm}_{\lambda}$ is defined with the Q-system $\Theta$ and the braiding on $\Rep(\cB)$.

If we have a finite group $G < \Aut(\cA)$, we can construct a subnet by taking its fixed point net $\cA^G$. A net obtained in this way is called an orbifold. Orbifolds of conformal nets and their representations have been studied in~\cite{Xu00} and~\cite{LX04,KLX05}.
To study the categorical structure of $\Rep(\cA^G)$ more systematically, M\"uger introduced the category of $G$-twisted representation $\GLoc \cA$ in~\cite{Mueg05}.
This category consists of $g$-localized transportable endomorphisms for every $g\in G$ as its objects, in addition to DHR endomorphisms.
In the same article~\cite{Mueg05}, it has been shown that $\GLoc \cA$ has a structure of braided $G$-crossed category, which is a monoidal category with a group action of $G$ and with a certain kind of braiding (see Definition~\ref{def_braided_G_crossed_category}).
Also, the relation between $\GLoc \cA$ and $\Rep(\cA^G)$ was clarified:
There exists a braided equivalence $(\GLoc \cA)^G \cong \Rep(\cA^G)$.
Moreover, there exists a equivalence of braided $G$-crossed categories $\Rep(\cA^G)\rtimes \Rep(G)\cong \GLoc \cA$ (see~\cite{Mueg05}, for notations and terminology which are not explained here).
Thus the study of $\GLoc \cA$ leads to the study of $\Rep(\cA^G)$.

In this article, we consider a situation that we have a given inclusion of conformal nets $\cB\subset \cA$ and a group $G < \Aut(\cA)$ which preserves $\cB$ globally.
Let us denote by $G' < \Aut(\cB)$ the group obtained by restricting each element of $G$ to $\cB$.
For such a situation, it is natural to study the relation between the categories $\GLoc \cA$ and $\GpLoc \cB$.
An orbifold net $\cA^G \subset \cA$ gives us an example of such a situation. In this case, we have $G' = \{e\}$ and $\GpLoc \cB$ is the same as $\Rep(\cA)$.

Our question is how to capture the braided $G$-crossed category $\GLoc \cA$ in terms of the algebraic structure on $\cB$.
To answer this, we make the following observation. The category $\GpLoc \cB$ and the Q-system $\Theta$ associated with $\cB\subset \cA$ are not enough to recover the category $\GLoc \cA$, since in general $G'$ does not remember the original group $G$. Even if $G'$ is isomorphic to $G$, one cannot determine the position of $G$ in $\Aut(\cA)$ in general.
For our purpose, it is desirable to describe $G$ (and its position in $\Aut(\cA)$) by some algebraic structure related to $\cB$. This task is achieved by using a notion of $G$-equivariant Q-system structures (see Section~\ref{section:G_action_on_subfactors_G_equivariant_str}, for terminologies).
Let us explain this in detail.
Since $G$ also acts on $\cB$ by our assumption, we have the induced action of $G$ on $\Rep(\cB)$.
Then one can construct the canonical $G$-equivariant Q-system structure on $\Theta$.
Such a consideration has been already observed in~\cite[Section 6]{BJLP19} for the case $G\cong G'$.
We will summarize $G$-equivariant Q-system structures for the case of finite index subfactors with actions of $G$ in Section~\ref{section:G_action_on_subfactors_G_equivariant_str}.

Using $G$-equivariant Q-system structure and the braiding, we introduce two types of induced endomorphisms for objects in $\GpLoc \cB$.
These are defined by similar formulas as $\alpha$-induced endomorphisms. One is defined only with the opposite braiding on $\GpLoc \cB$ and the Q-system $\Theta$, but the other is defined with the braiding on $\GpLoc \cB$, the Q-system $\Theta$, and the $G$-equivariant structure on $\Theta$ as explained above.
We study their properties and derive some formulas as in the case of $\alpha$-induction.
We will see that many statements for $\alpha$-induction have natural translations for our setting.

This article is organized as follows.
In Section 2, we start from some preliminaries on C*-tensor categories of endomorphisms, subfactors, Q-systems. We then discuss $G$-equivariant structures on Q-systems. We also recall the definition of braided $G$-crossed categories.
In Section 3, we collect some preliminaries on conformal nets and its twisted representations. For later applications, we give proofs of some statements on twisted representations.
In Section 4, we introduce the induced endomorphisms for twisted representations and study their properties in a similar strategy as the case of $\alpha$-induction.
We explain our setting and notations in Section~\ref{subsection:ind_for_twisted_reps_assumptions}.
Then we define two induced endomorphisms for an object in $\GpLoc \cB$ in Section~\ref{subsection:ind_for_twisted_reps_def}.
In the rest of this article, we study their properties.
We first summarize basic properties and study intertwiner spaces between induced endomorphisms. We also discuss the relation to $\sigma$-restriction and derive a version of the $\alpha\sigma$-reciprocity formula. Finally, we show that the $G$-crossed braiding of $\GLoc \cA$ is recovered from the $G'$-crossed braiding of $\GpLoc \cB$.

\section{Preliminaries}

\subsection{C*-tensor categories of endomorphisms}

In this subsection, we review the C*-tensor category structure of endomorphisms of type III factors and related terminologies. We refer the reader to \cite[Chapter 2]{NTBook} for the basics of C*-tensor categories. By a C*-tensor category $\cC$, we always assume that $\cC$ is strict, closed under finite direct sums and subobjects. We also assume that the tensor unit $\unit\in \cC$ is simple, i.e. $\End(\unit) = \CC$. 

For our purpose, we also consider homomorphisms between type III factors in addition to endomorphisms. Let $N, M$ be type III factors and $\rho, \sigma\colon N\to M$ a pair of homomorphisms. (Unless otherwise said, the notion of a homomorphism means a unital $*$-preserving homomorphism.) The space of intertwiners from $\rho$ to $\sigma$ is defined by
\[
\Hom(\rho, \sigma) = \{t\in M : t\rho(n) = \sigma(n)t \enspace \textrm{for all} \ n\in N\}.
\]
Clearly, $\Hom(\rho, \sigma)$ is a $\CC$--linear space. We write $\langle \rho, \sigma \rangle := \dim\Hom(\rho, \sigma)$ for its dimension. The concatenation product of intertwiners $t_1\in \Hom(\rho_1, \rho_2)$ and $t_2\in\Hom(\rho_2, \rho_3)$ is defined by the multiplication of operators $t_2\circ t_1 := t_2\cdot t_1 \in\Hom(\rho_1, \rho_3)$. The tensor product of homomorphisms $\rho\colon N\to M$, $\sigma\colon L\to N$ is the composition $\rho\otimes\sigma := \rho\circ\sigma = \rho\sigma$. Let $\rho_1, \sigma_1\colon N\to M$ and $\rho_2, \sigma_2\colon L\to N$ be homomorphisms. The tensor product of intertwiners $t_1\in \Hom(\rho_1, \sigma_1)$ and $t_2 \in\Hom(\rho_2, \sigma_2)$ is
\[
t_1\otimes t_2 := t_1\rho_1(t_2) = \sigma_1(t_2)t_1 \in\Hom(\rho_1\rho_2, \sigma_1\sigma_2).
\]
With these structures, one can consider a 2--C*-category $\cC$ whose objects are a set of type III factors $\cC=\{N, M, ... \}$, the 1--morphisms are homomorphisms, and the 2--morphisms are intertwiners. If one considers the case $\cC = \{N\}$, the set of endomorphisms of $N$, denoted by $\End(N)$, has a C*-tensor category structure.

\begin{remark} \label{remark:non_closed_alg_tensor_cat_of_endomorphisms}
If $A\subset \cB(\cH)$ is a unital $*$-subalgebra of $\cB(\cH)$ on some Hilbert space, one can view $\End(A)$ as a $\CC$-linear monoidal category with positive $*$-operation as above. Note that $\End(A)$ is not a C*-tensor category in general. 
\end{remark}

We call $\rho\colon N\to M$ \emph{irreducible} if $\Hom(\rho,\rho)=\rho(N)'\cap M = \CC\cdot 1_M$. 
Two homomorphisms $\rho,\sigma\colon N\to M$ are called \emph{unitarily equivalent} if there exists a unitary $u\in \Hom(\rho,\sigma)$. The equivalence class of $\rho$ is called a \emph{sector} and denoted by $[\rho]$. 
The \emph{direct sum} of a finite family of homomorphisms $\{\rho_i\colon N\to M\}_{i=1}^n$ is
\[
\bigoplus_{i=1}^n \rho_i := \sum_i s_i\rho_i(\cdot) s_i^*
\]
where $s_i$ are isometries in $M$ with $\sum_i s_i s_i^* = 1$. Note that the direct sum is defined up to unitary equivalence. A homomorphism $\sigma$ is called a \emph{subobject} of $\rho$, denoted by $\sigma\prec\rho$, if there exists an isometry $s\in \Hom(\sigma,\rho)$. 

A homomorphism $\overline{\rho}\colon M\to N$ are said to be a \emph{conjugate} of $\rho\colon N\to M$ if there is a pair of intertwiners $R\in \Hom(\id_N,\overline{\rho}\rho)$ and $\overline{R}\in\Hom(\id_M,\rho\overline{\rho})$ satisfying the following \emph{conjugate equation}:
\begin{subequations}
\label{conjugate_eq}
\begin{align}
(1_\rho \otimes R^*)\cdot (\overline{R}\otimes 1_\rho) &\equiv \rho(R^*)\overline{R} = 1_\rho,  \label{conjugate_eq1} \\
(1_{\overline{\rho}} \otimes \overline{R}^*)\cdot (R \otimes 1_{\overline{\rho}}) &\equiv \overline{\rho}(\overline{R}^*)R = 1_{\overline{\rho}}. \label{conjugate_eq2}
\end{align}
\end{subequations}
The \emph{statistical dimension} $d\rho$ of $\rho$ is defined by the minimum of the numbers $\|R\|\cdot\|\overline{R}\|$ over all solutions. A solution $(R, \overline{R})$ of the conjugate equation is called \emph{standard} if $\|R\|=\|\overline{R}\|=d\rho^{1/2}$. The conjugate and the standard solution are unique up to unitary in the following sense: If $\overline{\rho}'\colon M\to N$ and $(R', \overline{R}')$ is another standard solution of the conjugate equation for $\rho$ and $\overline{\rho}'$, then there exists a unitary $u\in \Hom(\overline{\rho}, \overline{\rho}')$ such that
\begin{equation} \label{unitary_another_conjugate}
R' = (u\otimes 1_\rho)\cdot R \equiv uR \quad \text{and} \quad \overline{R}' = (1_\rho\otimes u)\cdot \overline{R}\equiv \rho(u)\overline{R}.
\end{equation}
For the convenience, we set $d\rho = \infty$ if $\rho$ does not have a conjugate. It is well known~\cite{LonI89, LonII90} that the statistical dimension is equal to the square root of the minimum index~\cite{Hiai88} of $\rho(N)\subset M$:
\begin{equation*}
d\rho = [M: \rho(N)]^{1/2}.
\end{equation*}

\subsection{Subfactors and Q-systems}

We briefly review the notion of Q-systems introduced by Longo~\cite{Lon94} and its relation to subfactors.

\begin{definition}\label{def_Q_system}
Let $\cC$ be a C*-tensor category.
A triple $\Theta = (\theta, w, x)$ with $\theta\in\cC$ and intertwiners $w\in\Hom(\id_N, \theta)$ and $x\in\Hom(\theta, \theta\otimes \theta)$ is called a \emph{Q-system} in $\cC$ if it satisfies 
\begin{alignat*}{2}
(x\otimes 1_\theta)x &= (1_\theta\otimes x)x, & \qquad & \text{(associativity)}\\
(w^*\otimes 1_\theta)x &= (1_\theta\otimes w^*)x = 1_\theta. & & \text{(unit law)}
\end{alignat*}
and
\begin{equation*} \label{Q_system_standardness}
w^*w = \sqrt{d\theta}\cdot 1_\id, \quad \text{and} \quad x^*x = \sqrt{d\theta}\cdot 1_\theta. \qquad \text{(standardness)}
\end{equation*}
A Q-system is called \emph{irreducible} if $\dim\Hom(\id_N, \theta) = 1$.

Two Q-systems $\Theta = (\theta, w, x)$ and $\Theta' = (\theta', w', x')$ in $\cC$ are called \emph{equivalent}, if there is a unitary $u\in\Hom(\theta, \theta')$ such that
\begin{align}
x'u &= (u\otimes u)x, \\
w' &= uw.
\end{align}
We call such unitary a \emph{unitary isomorphism of Q-systems}.
\end{definition}

\begin{remark}
\begin{enumerate}
\item It seems that there is no agreement on the definition of Q-systems. In some literatures, the triple as above is called a standard Q-system.
\item A Q-system $(\theta, w, x)$ automatically satisfies the Frobenius property~\cite{LR97}:
\begin{equation*}
(1_\theta\otimes x^*)\circ(x\otimes 1_\theta) = x\circ x^* = (x^*\otimes 1_\theta)\circ(1_\theta\otimes x).
\end{equation*}
Thus $(\theta, w, x)$ is a standard C*-Frobenius algebra object in $\cC$.
\end{enumerate}
\end{remark}

Let $N\subset M$ be a finite index type III subfactor with the inclusion map $\iota\colon N\hookrightarrow M$. Then $N\subset M$ gives a Q-system in $\End(N)$ as follows. Since we assume $[M: N] < \infty$, there is a conjugate $\iotabar: M\to N$ of $\iota$. Let $w\in\Hom(\id_N, \iotabar\iota)$ and $v\in\Hom(\id_M, \iota\iotabar)$ be a standard solution of the conjugate equation Eq.~\eqref{conjugate_eq} for $\iota$ and $\iotabar$. Then, it is easily checked that the triple $\Theta = (\theta, w, x)$ with
\begin{equation}  \label{Qsystem_from_subfactor}
\theta:=\iotabar\iota\in \End(N), \quad w\in \Hom(\id_N, \theta), \quad x:=\iotabar(v)\equiv 1_{\iotabar}\otimes v\otimes 1_{\iota}\in\Hom(\theta, \theta^2),
\end{equation}
is a Q-system. If $N\subset M$ is irreducible, then the corresponding Q-system $\Theta$ is also irreducible.
The endomorphism $\theta$ is called the \emph{dual canonical endomorphism} of $N\subset M$.
(The endomorphism $\iota\iotabar\in \End(M)$ is called the \emph{canonical endomorphism} of $N\subset M$.)
The map $E(\cdot)=\frac{1}{d\iota} w^*\overline{\iota}(\cdot)w\colon M\to N$ gives a conditional expectation from $M$ to $N$. By using the conjugate equation, it is easily seen that we have the following formula
\begin{equation} \label{subfactor_basis_expansion}
	m = d\iota \cdot E(mv^*)v, \quad \text{for} \ m\in M,
\end{equation}
and hence we have $M=Nv$.
\begin{remark} \label{remark:subfactor_expansion_uniqueness}
Using the conjugate equation, it is also easily seen  that $nv = 0$ implies $n=0$ for $n\in N$ (cf. \cite[Lemma 3.8]{BE1}).
\end{remark}

\begin{remark} \label{remark:Q_system_another_choice}
Let $\iotabar'\colon M\to N$ be another choice of a conjugate of $\iota$ and $(w', v')$ a standard solution of the conjugate equation for $\iota$ and $\iotabar'$.
We denote by $\Theta' = (\theta'=\iotabar'\iota, w', x'=\iotabar'(v'))$ the corresponding Q-system as above.
Since the conjugate and the associated standard solution is unique up to unitary,  there exists a unitary $u\in \Hom(\iotabar, \iotabar')$ satisfying
\begin{equation} \label{eq:iota_conjugate_uniqueness}
    w' = uw, \quad v' = \iota(u)v = uv,
\end{equation}
as in Eq.~\eqref{unitary_another_conjugate}.
One can check that $u = u\otimes 1_\iota\in \Hom(\theta, \theta')$ is a unitary isomorphism of Q-systems from $\Theta$ to $\Theta'$.
\end{remark}

Conversely, starting from an irreducible Q-system $\Theta$ in $\End(N)$, one can construct an irreducible overfactor $M\supset N$ such that $\Theta$ is a Q-system obtained as above~\cite{Lon94} (also cf.~\cite{BKLR15}). Therefore there is a one-to-one correspondence between irreducible Q-systems in $\End(N)$ up to equivalence and irreducible finite index subfactors $N\subset M$ up to conjugation.

\subsection{Braided C*-tensor categories and $\alpha$-induction}

Let $\cC$ be a C*-tensor category. A family of natural isomorphisms $\{\eps(\lambda, \mu)\colon \rho\otimes\mu \to \mu\otimes\lambda \}_{\lambda, \mu\in \cC}$ is called a \emph{braiding} on $\cC$ if it satisfies the usual hexagon identities (cf. \cite{EGNOBook}). In this article, we always assume that each $\eps(\lambda, \mu)$ is a unitary. For a braiding $\{\eps(\lambda, \mu) \}_{\lambda, \mu\in \cC}$, its opposite braiding is defined by $\eps^-(\lambda, \mu)=\eps(\mu, \lambda)^*$ and the family $\{\eps^-(\lambda, \mu)\}_{\lambda, \mu\in \cC}$ is also a braiding on $\cC$. We also write $\eps^+(\lambda, \mu)\equiv \eps(\lambda, \mu)$. A C*-tensor category equipped with a braiding is called a \emph{braided C*-tensor category}.

A Q-system $\Theta=(\theta, w, x)$ in a braided C*-tensor category $\cC$ is said to be \emph{commutative} if it satisfies $\eps(\theta, \theta)x = x$.

A braiding on a C*-tensor category $\cC$ is called \emph{non-degenerate} if the M\"uger center
\begin{equation*}
\cC \cap \cC' = \{\lambda\in \cC : \eps(\lambda, \mu)\eps(\mu, \lambda) = 1_{\mu\otimes \lambda} \ \text{for all} \ \mu\in \cC \}
\end{equation*}
is trivial, i.e. each objects in $\cC\cap \cC'$ is a finite direct sum of tensor unit.
A rigid braided C*-tensor category $\cC$ is called \emph{unitary modular tensor category} if there are only finitely many inequivalent irreducible objects and its braiding is non-degenerate.

Let $N\subset M$ be a finite index type III subfactor and $\cC$ a braided C*-tensor category realized as a full and replete subcategory of $\End(N)$. 
Suppose that the dual canonical endomorphism $\theta=\iotabar\iota$ of $N\subset M$ is in $\cC$.
For $\lambda\in \cC$, its \emph{$\alpha$-induction} is defined by
\begin{equation} \label{eq:def_of_alpha_induction}
\pmalpha{\lambda} = \iotabar^{-1}\circ \Ad(\eps^{\pm}(\lambda, \theta))\circ \lambda\circ\iotabar \in \End(M).
\end{equation}
The notion of $\alpha$-induction was first introduced in \cite{LR95}, and further studied in~\cite{Xu98} and~\cite{BE1,BE2,BE3} for nets of subfactors.

\subsection{Group actions on subfactors and $G$-equivariant Q-systems} \label{section:G_action_on_subfactors_G_equivariant_str}

Let us fix a group $G$ in this subsection. We summarize a relation between group actions on subfactors and $G$-equivariant structure on the associated Q-systems.

We first recall some terminologies.
By a \emph{strict action of $G$} on a strict monoidal category $\cC$, we mean a group homomorphism $\gamma\colon G \to \Aut^{\mathrm{strict}}_\otimes(\cC), g\mapsto \gamma_g$, where $\Aut^{\mathrm{strict}}_\otimes(\cC)$ denotes the group of all strict monoidal automorphisms of $\cC$.

For simplicity, we give the following definitions under certain strictness conditions for both categories and group actions.
\begin{definition}
Let $\gamma$ be a strict $G$-action on a strict monoidal category $\cC$.
A \emph{$G$-equivariant object} is a pair $(X, z)$ consisting of an object $X\in\cC$ and a family of isomorphisms $z = \{z_g\colon \gamma_g(X) \to X\}_{g\in G}$ such that the following diagram
\begin{equation} \label{G_eq_obj_diagram}
\xymatrix{
	\gamma_g(\gamma_h(X)) \ar[rr]^{\gamma_g(z_h)} \ar[drr]_{z_{gh}} && \gamma_g(X) \ar[d]^{z_g} \\
	&& X 
}
\end{equation}
commutes for all $g, h\in G$.

For two $G$-equivariant objects $(X, z)$ and $(X', z')$, a morphism $f\colon X \to X'$ is called a \emph{$G$-equivariant morphism} if the following diagram 
\begin{equation} \label{G_eq_morphism_diagram}
\xymatrix{
	\gamma_g(X) \ar[d]_{\gamma_g(f)} \ar[rr]^{z_g}&& X \ar[d]^f \\
	\gamma_g(X') \ar[rr]_{z'_g} && X'
}
\end{equation}
commutes for all $g\in G$.
\end{definition}

We now consider the C*-tensor category of endomorphisms.
Let $N$ be a type III factor and $\beta$ an action of $G$ on $N$. Then $\beta$ induces a strict action $\gamma$ of $G$ on $\End(N)$ by
\begin{alignat*}{2}
\gamma_g(\rho) &= \beta_g\circ\rho\circ\beta_g^{-1} &\qquad &\text{for} \ \rho\in\End(N), \\
\gamma_g(s) &= \beta_g(s) & &\text{for} \ s\in \Hom(\rho, \sigma).
\end{alignat*}
Let $\Theta = (\theta, w, x)$ be a Q-system in $\End(N)$. Then $g(\Theta) := (\gamma_g(\theta), \gamma_g(w), \gamma_g(x))$ is also a Q-system in $\End(N)$.
Assume that there exists a family of unitaries $z = \{z_g\colon \gamma_g(\theta)\to \theta\}_{g\in G}$. If $z_g$ is a unitary isomorphism of Q-system for each $g\in G$ and the pair $(\theta, z)$ is a $G$-equivariant object, then we call $(\Theta, z)$ a \emph{$G$-equivariant Q-system}.

Suppose that we have also an extension $M$ of $N$ such that $N\subset M$ is a finite index type III subfactor.
Let $\Theta=(\theta=\iotabar\iota, w, x=\iotabar(v))$ be a Q-system in $\End(N)$ associated with $N\subset M$.
We will see below that there exists a one-to-one correspondence between the following two.
\begin{itemize}
\item actions of $G$ on $M$ which extend $\beta$ (or equivalently, actions of $G$ on the subfactor $N\subset M$ which act as $\beta$ on $N$).
\item families of unitaries which define $G$-equivariant Q-system structures on $\Theta$.
\end{itemize}

\begin{remark}
We note that such a correspondence has been established in~\cite[Section 6]{BJLP19} for finite index inclusions of completely rational conformal nets $\cB\subset \cA$ and $G < \Aut(\cB)$ (see Section~\ref{section:conformal_nets_twisted_reps} for conformal nets).
In the same article, $G$-equivariant algebra structures on algebra objects in tensor categories (not necessarily C*) have been studied to understand spontaneous symmetry breaking under anyon condensation.

For our later applications, we will see this correspondence with a slightly explicit manner. Note that finiteness of $G$ is not needed in the following argument.
\end{remark}

We first see that one can construct a structure of $G$-equivariant Q-system from an extension of group action.

\begin{lemma} \label{lemma:G_equivariant_str_on_Q_system}
Under the assumptions as above, suppose that we have an action $\tilbeta$ of $G$ on $M$ which extends $\beta$, i.e. $\tilbeta_g\iota = \iota\beta_g$ for all $g\in G$.
For each $g\in G$, we define $z_g\in N$ by the unique element satisfying
\begin{equation}
\tilbeta_g(v) = z_g^* v.
\end{equation}
Then the following hold:
\begin{enumerate}
	\item $z_g\in \Hom(\beta_g\iotabar, \iotabar\tilbeta_g)$
	\item $z_g$ is a unitary isomorphisms of Q-systems from $g(\Theta)$ to $\Theta$.
	\item $(\theta, z)$ is a $G$-equivariant object in $\End(N)$ with strict action $\gamma$ of $G$ where $z$ denotes the family of unitaries $\{z_g\colon \gamma_g(\theta)\to \theta\}_{g\in G}$.
	\item $\tilbeta_g = \iotabar^{-1}\circ \Ad(z_g)\circ \beta_g\circ\iotabar$.
\end{enumerate}
In particular, $(\Theta, z)$ is a $G$-equivariant Q-system in $\End(N)$.

Moreover, the above $G$-equivariant Q-system structure only depends on $N\subset M$ and $\tilbeta$ in the following sense:
Let $\Theta'=(\theta'=\iotabar\iota, w', x'=\iotabar'(v'))$ be another Q-system associated with $N\subset M$ and $u\in \Hom(\iotabar, \iotabar')$ be a unitary as in Remark~\ref{remark:Q_system_another_choice}. We denote by $z' = \{z'_g\colon \gamma_g(\theta')\to \theta'\}_{g\in G}$ the corresponding family of unitaries as above. Then $u$ is a $G$-equivariant morphism from $(\theta, z)$ to $(\theta', z')$.
\end{lemma}

\begin{remark}
Some parts of (i) and (ii) have been essentially observed in~\cite[Lemma 6.1.1]{BisThesis} and its proof.
For the readers' convenience, we give the proof of all statements here.
\end{remark}

\begin{proof}
We first claim that $\iotabar_g := \beta_g\iotabar\tilbeta_g^{-1}\colon M\to N$ is a conjugate of $\iota$ and $(w_g, v_g) := (\beta_{g}(w), \tilbeta_g(v))$ is a standard solution for $\iota$ and $\iotabar_g$ for each $g\in G$.
Using the fact that $(w, v)$ is a standard solution of conjugate equation for $\iota$ and $\iotabar$, one can check this claim by direct computation. By Remark~\ref{remark:Q_system_another_choice}, there exists a unitary $u_g\in \Hom(\iotabar_g, \iotabar)$ such that $\iota(u_g)v_g = v$ and $u_g w_g = w$. 
Moreover, $u_g = u_g\otimes 1_\iota \in \Hom(\iotabar_g\iota, \theta)$ is a unitary isomorphism of Q-system from $(\iotabar_g\iota, w_g, \iotabar_g(v_g))$.
Since $u_g^*v = v_g\equiv \tilbeta_g(v)= z_g^*v$ and $u_g, z_g\in N$, we have $z_g = u_g$ by Remark~\ref{remark:subfactor_expansion_uniqueness}. In particular, $z_g$ is a unitary.

\smallskip

(i) Since $z_g = u_g\in \Hom(\beta_{g}\iotabar\tilbeta_g^{-1}, \iotabar)$, we get $z_g\in \Hom(\beta_g\iotabar, \iotabar\tilbeta_g)$ by composing $\tilbeta_g$ from the right.

\smallskip

(ii) We note that the Q-system $(\iotabar_g\iota, w_g, \iotabar_g(v_g))$ obtained from the conjugate $\iotabar_g\colon M\to N$ and the standard solution $(w_g, v_g)$ is equal to $g(\Theta)$. In fact, we have
\begin{align*}
\iotabar_g\iota &= \beta_g\iotabar\tilbeta_g^{-1}\iota = \beta_g\iotabar\iota\beta_{g}^{-1} = \gamma_g(\theta), \\
w_g &= \beta_g(w) = \gamma_g(w), \\
\iotabar_g(v_g) &= \beta_g\iotabar\tilbeta_g^{-1}(\tilbeta_g(v)) = \beta_g(\iotabar(v)) = \beta_g(x) = \gamma_g(x).
\end{align*}
Since $z_g=u_g=u_g\otimes 1_\iota$ is a unitary isomorphism of Q-system from $(\iotabar_g\iota, w_g, \iotabar_g(v_g))$ to $\Theta$, we obtain the statement.

\smallskip

(iii) It suffices to show $z_{gh} = z_g\gamma_g(z_h)\equiv z_g\beta_{g}(z_h)$ for every $g, h\in G$. To see this, we have
\begin{align*}
z_{gh}^*v = \tilbeta_{gh}(v) = \tilbeta_g(z_h^*v) = \beta_g(z_h^*)z_g^*v.
\end{align*}
Thus we get $z_{gh} = z_g\beta_{g}(z_h)$ by Remark~\ref{remark:subfactor_expansion_uniqueness}.

\smallskip

(iv) For every $m\in M$, we have
\begin{equation*}
    \Ad(z_g)\circ \beta_g\circ\iotabar(m) = \iotabar\circ \tilbeta_g (m),
\end{equation*}
where we used the intertwining property (i). (This computation also implies the well-definedness of the formula.) Applying $\iotabar^{-1}$ to both sides, we get $\tilbeta_g(m) = \iotabar^{-1}\circ \Ad(z_g)\circ \beta_g\circ \iotabar(m)$.

\smallskip

Finally, we prove the last statement. Let $g\in G$. We compute $\tilbeta_g(v')$ in two ways:
\begin{align*}
\tilbeta_g(v') &= {z'_g}^* v' = {z'_g}^* uv, \\
\tilbeta_g(v') &= \tilbeta_g(uv) = \beta_g(u)z_g^* v.
\end{align*}
Thus we get ${z'_g}^*u = \beta_g(u)z_g^* = \gamma_g(u)z_g^*$ and hence $z'_g \gamma_g(u) = u z_g$, which implies the commutativity of the diagram~\eqref{G_eq_morphism_diagram}. This proves the statement.
\end{proof}

We then see the converse direction.

\begin{lemma}
Under the assumptions as above, suppose that we have a family of unitaries $z = \{z_g\colon \gamma_g(\theta)\to \theta \}$ which gives a structure of a $G$-equivariant Q-system on $\Theta$. Then the map $G\ni g\mapsto \tilbeta_g\in \Aut(M)$ defined by
\begin{equation*}
\tilbeta_g = \iotabar^{-1}\circ \Ad(z_g)\circ \beta_g\circ \iotabar
\end{equation*}
gives a well-defined group action of $G$ on $M$ which extends the action $\beta$ on $N$.
\end{lemma}

\begin{proof}
The well-definedness of $\tilbeta_g$ and the claim that $\tilbeta_g$ is an extension of $\beta_g$ for each $g\in G$ have been shown in~\cite[Proposition 6.1.3]{BisThesis} under a similar setting.
To see that $\tilbeta$ is a group homomorphism, for $g, h\in G$ we have
\begin{align*}
\tilbeta_g\circ \tilbeta_h &= \iotabar^{-1}\circ \Ad(z_g)\circ \beta_g\circ \Ad(z_h)\circ \beta_h\circ \iotabar \\
&= \iotabar^{-1}\circ \Ad(z_g\beta_g(z_h))\circ \beta_g\circ\beta_h\circ \iotabar \\
&= \iotabar^{-1}\circ \Ad(z_{gh})\circ \beta_{gh} \circ\iotabar = \tilbeta_{gh},
\end{align*}
where we used $z_{gh} = z_g\gamma_g(z_h) = z_g\beta_g(z_h)$.
\end{proof}

\subsection{Braided $G$-crossed categories} \label{section:braided_G_crossed_category}

Let $G$ be a group. We recall the notion of braided $G$-crossed categories and introduce its ''opposite" braiding. For simplicity, we only give the definition with strictness assumption.

\begin{definition} \label{def_braided_G_crossed_category}
A \emph{strict $G$-crossed category} is a strict monoidal category $\cF$ equipped with the following structure:
\begin{itemize}
	\item There is a full subcategory $\cF_{\mathrm{hom}}\subset \cF$ of homogeneous objects.
	\item $\cF_{\mathrm{hom}}$ is $G$-graded, in the sense that there is a map $\partial\colon \cF_{\mathrm{hom}}\to G$ such that $\partial(X\otimes Y) = \partial X \cdot \partial Y$ and $\partial X = \partial Y$ if $X\cong Y$. We define the full subcategories $\cF_g$ by $\cF_g = \partial^{-1}(g)$.
	\item There is a strict $G$-action $\gamma$ on $\cF$ such that $\gamma_g(\cF_h) \subset\cF_{ghg^{-1}}$.
\end{itemize}
If $\cF$ is additive, we require that every object in $\cF$ is a finite direct sum of objects in $\cF_{\mathrm{hom}}$, i.e. $\cF = \bigoplus_{g\in G} \cF_g$. The $G$-grading is said to be \emph{full} if the image of $\partial$ is equal to $G$. 

A \emph{$G$-crossed braiding} (or simply \emph{braiding}) on a strict $G$-crossed category $\cF$ is a family of isomorphisms
\begin{equation*}
c(X, Y): X\otimes Y \to \gamma_g(Y)\otimes X,
\end{equation*}
defined for all $g\in G$, $X\in \cF_g$ and $Y\in \cF$ satisfying following identities:

\begin{enumerate}
	\item (Naturality) For all $g\in G$, $X, X'\in\cF_g$, $Y, Y'\in\cF$, $s\in\Hom(X, X')$, and $t\in\Hom(Y, Y')$ we have
	\begin{align} \label{naturality_gcbraiding}
	c(X', Y')\circ (s\otimes t) = (\gamma_g(t)\otimes s)\circ c(X, Y).
	\end{align}
	\item (Covariance) For all $g\in G$, $X\in\cF_{\text{hom}}$, and $Y\in\cF$ we have
	\begin{align} \label{covariance_gcbraiding}
	\gamma_g(c(X, Y)) = c(\gamma_g(X), \gamma_g(Y)).
	\end{align}
	\item For all $g, h\in G$, $X\in \cF_g$, $Y\in \cF_h$, and $Z\in \cF$ we have
	\begin{subequations} \label{braid_rel_gcbraiding}
	\begin{align}
	c_{X\otimes Y, Z} &= (c(X, \gamma_h(Z))\otimes 1_Y)\circ (1_X\otimes c(Y, Z)), \label{braid_rel_gcbraiding1} \\
	c_{X, Y\otimes Z} &= (1_{\gamma_g(Y)}\otimes c(X, Z))\circ (c(X, Y)\otimes 1_Z). \label{braid_rel_gcbraiding2}
	\end{align}
	\end{subequations}
\end{enumerate}
A strict $G$-crossed category equipped with a $G$-crossed braiding is called a \emph{strict braided $G$-crossed category}.
\end{definition}

Let $\cF$ be a braided $G$-crossed category. Using Eq.~\eqref{braid_rel_gcbraiding}, one can show a version of the Yang-Baxter equation:
For all $g, h\in G$, $X\in \cF_g$, $Y\in \cF_h$, and $Z\in \cF_{\mathrm{hom}}$, we have
\begin{multline} \label{eq:YBE_for_G_crossed_braiding}
(c(\gamma_g(Y), \gamma_g(Z))\otimes 1_X)\circ(1_{\gamma_g(Y)}\otimes c(X, Z))\circ(c(X, Y)\otimes 1_Z) = \\
(1_{\gamma_{gh}(Z)}\otimes c(X, Y))\circ (c(X, \gamma_h(Z))\otimes 1_Y)\circ (1_X\otimes c(Y, Z)).
\end{multline}
We note that the identity component $\cF_e$ is closed under the tensor product and restriction of a $G$-crossed braiding on $\cF_e$ gives an ordinary braiding. We also note that the strict action $\gamma$ of $G$ preserves $\cF_e$.
We refer the reader to the book~\cite{TuBook} for the basics of braided $G$-crossed categories.

For the later use, we introduce the ''opposite" version of $G$-crossed braiding. We define $c^-(X, Y)$ by
\begin{equation}
c^-(X, Y) := c(Y, \gamma_{h^{-1}}(X))^{-1}\colon X\otimes Y\to Y\otimes \gamma_{h^{-1}}(X)
\end{equation}
for all $X\in\cF_{\text{hom}}$ and $Y\in\cF_h$. Using the identities in the definition of $G$-crossed braiding, one can check that the family $\{c^-(X, Y)\}$ satisfies the following identities:
\begin{enumerate}
	\item (Naturality) For all $h\in G$, $X, X'\in\cF_{\text{hom}}$, $Y, Y'\in\cF_h$, $s\in\Hom(X, X')$, and $t\in\Hom(Y, Y')$ we have
	\begin{align} \label{naturality_gcopbraiding}
	c^-(X', Y')\circ (s\otimes t) = (t\otimes \gamma_{h^{-1}}(s))\circ c^-(X, Y).
	\end{align}
	\item (Covariance) For all $g\in G$, and $X, Y\in\cF_{\text{hom}}$ we have
	\begin{align} \label{covariance_gcopbraiding}
	\gamma_g(c^-(X, Y)) = c^-(\gamma_g(X), \gamma_g(Y)).
	\end{align}
	\item For all $h, k\in G$, $X\in \cF_{\mathrm{hom}}$, $Y\in \cF_h$, and $Z\in \cF_k$ we have
	\begin{subequations} \label{braid_rel_gcopbraiding}
		\begin{align}
		c^-(X\otimes Y, Z) &= (c^-(X, Z)\otimes 1_{\gamma_{k^{-1}}(Y)})\circ (1_X\otimes c^-(Y, Z)), \label{braid_rel_gcopbraiding1} \\
		c^-(X, Y\otimes Z) &= (1_Y\otimes c^-(\gamma_{h^{-1}}(X), Z))\circ (c^-(X, Y)\otimes 1_Z). \label{braid_rel_gcopbraiding2}
		\end{align}
	\end{subequations}
\end{enumerate}
Note that $c^-$ is not a $G$-crossed braiding. We will write $c^+(X, Y)\equiv c(X, Y)$ for the original $G$-crossed braiding.

\section{Conformal nets and twisted representations} \label{section:conformal_nets_twisted_reps}

\subsection{Conformal nets}

Let $S^1\subset \CC$ be the unit circle, which is identified with the one-point compactification of the real line $\RR\cup \{\infty\}$ by the Cayley transform $\RR\ni x \mapsto z=\frac{i - x}{i+x}$ where $\infty$ corresponds to $-1\in S^1$. We denote the M\"obius group by $\Mob$. This group is isomorphic to $\PSU(1, 1)$, which acts naturally on $S^1\subset \CC$.

We denote by $\cI$ the set of \emph{proper intervals} (or simply \emph{intervals}) of $S^1$, i.e. open, connected, non-dense, and non-empty subsets of $S^1$. For any subset $E\subset S^1$, we denote by $E'$ the interior of $S^1\setminus E$.

\begin{definition}
A \emph{local M\"obius covariant net} on $S^1$ is a quadruple $(\cA, U, \cH, \Omega)$ (simply denoted by $\cA$), where $\cH$ is a Hilbert space, $\Omega\in\cH$ is a unit vector corresponding to the vacuum, $U$ is a strongly continuous unitary representation of $\mathsf{M\ddot{o}b}$, and $\cA$ is a map
\[
\cI\ni I \mapsto \cA(I) \subset \cB(\cH)
\]
from the set of proper intervals to von Neumann algebras on $\cH$, with the following properties:
\begin{enumerate}
    \item (Isotony) For any pair of intervals $I_1, I_2\in \cI$, if $I_1 \subset I_2$, then 
    \[
        \cA(I_1) \subset \cA(I_2).
    \]
    \item (Locality) For any pair of intervals $I_1, I_2\in \cI$,if $I_1 \cap I_2 = \emptyset$, then 
    \[
        [\cA(I_1), \cA(I_2)] = \{0\},
    \]
    where brackets denote the commutator.
    \item (M\"obius covariance) For any $g\in \Mob$ and $I\in \cI$, we have 
    \[
        U(g)\cA(I)U(g)^* = \cA(gI).
    \]
	\item (Positive energy condition) The generator of the one-parameter rotation subgroup of $U$ has a positive spectrum.
	\item (Vacuum) $\Omega$ is the unique $U$-invariant vector (up to phase) and cyclic for the von Neumann algebra $\bigvee_{I\in\cI} \cA(I)$.
\end{enumerate}
If $\cA$ also satisfies the following diffeomorphism covariance in addition to the above axioms, it is called a \emph{conformal net}:
\begin{enumerate}
\setcounter{enumi}{5}
\item (Diffeomorphism covariance) The representation $U$ of $\Mob$ extends to a projective unitary representation of $\Diff^+(S^1)$ such that for all $I\in \cI$
\begin{align*}
U(g)\cA(I)U(g)^* &= \cA(gI), \quad \text{for all} \ g\in \Diff^+(S^1), \\
U(g)xU(g) &= x, \quad \text{for all} \ x\in \cA(I), \ g\in \Diff^+(I'),
\end{align*}
where $\Diff^+(S^1)$ denotes the group of orientation preserving diffeomorphism of $S^1$ and $\Diff^+(I')$ the subgroup of $\Diff^+(S^1)$ consisting of those $g$ such that $g(z) = z$ for all $z\in I$.
\end{enumerate}
\end{definition}

Let $\cA$ be a local M\"obius covariant net. We collect some properties of $\cA$ which automatically follow from the above axioms:
\begin{itemize}
    \item (Irreducibility)~\cite{GL96} The von Neumann algebra $\bigvee_{I\in\cI} \cA(I)$ is equal to $\cB(\cH)$.
	\item (Reeh-Schlieder property)~\cite{FJ96} The vacuum vector $\Omega$ is cyclic and separating for each local algebra $\cA(I)$, $I\in \cI$.
	\item (Bisognano-Wichmann property)~\cite{BGL93,GF93} For every $I\in \cI$, there exists a one-parameter subgroup $\{\Lambda_I(t)\}\subset \Mob$ which fixes the boundary points of $I$ such that
	\begin{equation*}
        \Delta_I^{it} = U(\Lambda_I(-2\pi t))
    \end{equation*}
    where $\Delta_I$ denotes the modular operator associated to the pair $(\cA(I), \Omega)$. Moreover, the modular conjugation $J_I$ associated to the pair $(\cA(I), \Omega)$ has also a geometric meaning.
	\item (Haag duality)~\cite{BGL93,GF93} $\cA(I') = \cA(I)'$ for any $I\in \cI$.
	\item (Additivity)~\cite{FJ96} If a family of intervals $\{I_i \in \cI\}$ and an interval $I\in \cI$ satisfies $I\subset \bigcup_i I_i \in \cI$, then we have $\cA(I)\subset \bigvee_i \cA(I_i)$.
	\item  (Factoriality) (see e.g. \cite{GF93}) Each local algebra $\cA(I)$ is a type III$_1$ factor.
\end{itemize}
We then recall some extra properties of a local M\"obius covariant net.
The net $\cA$ is said to be \emph{strongly additive} if we have $\cA(I_1)\vee\cA(I_2) = \cA(I)$ for any two adjacent intervals $I_1, I_2\in\cI$ obtained by removing a single point from an interval $I\in\cI$. 
We say that $\cA$ has the \emph{split property} if $\cA(I_1)\vee \cA(I_2)$ is naturally isomorphic to $\cA(I_1)\otimes \cA(I_2)$ for any pair of intervals $I_1, I_2\in \cI$ with disjoint closures.

\begin{definition}[\cite{KLM01}]
A local M\"obius covariant net $\cA$ is called \emph{completely rational} if it is strongly additive, fulfills the split property and its $\mu$-index $\mu$$_\cA$ is finite, where
\[
\mu_\cA = [(\cA(I_2)\vee \cA(I_4))': \cA(I_1)\vee \cA(I_3)]
\]
for $I_1, I_3\in \cI$ two intervals with $\overline{I_1}\cap\overline{I_3} = \emptyset$ and $I_2, I_4\in\cI$ the two components of $(I_1\cup I_3)'$ (which does not depend on the choice of intervals).
\end{definition}

\subsection{Subnets}

Let $(\cA, U, \cH, \Omega)$ be a local M\"obius covariant net. A \emph{M\"obius covariant subnet} $\cB$ of $\cA$ is a map
\[
\cI \ni I \mapsto \cB(I)\subset \cB(\cH)
\]
which associates to each $I\in \cI$ a von Neumann subalgebra $\cB(I)$ of $\cA(I)$, satisfying the following
\begin{alignat*}{2}
\cB(I_1) &\subset \cB(I_2) \qquad &&\text{if} \ I_1 \subset I_2, \\
U(g)\cA(I)U(g)^* &= \cA(gI) &&\text{for all} \ g\in \Mob, \ I\in \cI.
\end{alignat*}
We write $\cB\subset \cA$ to indicate that $\cB$ is a M\"obius covariant subnet of $\cA$.

Let $\cH_\cB$ be the closure of $\bigvee_I \cB(I)\Omega$. Then $\cH_\cB$ is $U$-invariant and the restriction of $\cB$ to $\cH_\cB$ gives a local M\"obius covariant net on $\cH_\cB$. The restriction map $\cB(I)\ni b\mapsto b|_{\cH_\cB}$ is injective since $\Omega$ is separating for $\cB(I)$, so we will often identify $\cB$ with its restriction to $\cH_\cB$.

By the M\"obius covariance, the index $[\cA(I):\cB(I)]$ does not depend on the interval $I\in \cI$ and will be denoted by $[\cA:\cB]$. By~\cite[Corollary 2.6]{DLR01}, $\cB(I)'\cap \cA(I) = \CC$ for all $I\in \cI$ if $[\cA:\cB] < \infty$.

\subsection{Automorphism groups and orbifolds}

Let $\cA$ be a local M\"obius covariant net. The \emph{automorphism group} of $\cA$ is defined by
\[
\Aut(\cA) = \{g\in\cU(\cH) \mid g\cA(I)g^{-1} = \cA(I) \ \textrm{for all} \ I\in\cI \ \textrm{and} \ g\Omega = \Omega \},
\]
where $\cU(\cH)$ denotes the group of unitaries on $\cH$. We consider $\Aut(\cA)$ as a topological group with the topology induced by the strong operator topology of $\cB(\cH)$.
By definition, every $g\in \Aut(\cA)$ gives an automorphism $\Ad (g)|_{\cA(I)}$ of each local algebra $\cA(I)$. It is easy to see that every element of $\Aut(\cA)$ commutes with the representation $U$ of $\Mob$ by using the Bisognano-Wichmann property. If $\Ad (g)|_{\cA(I)} = \id_{\cA(I)}$ for some interval $I\in \cI$, then we have $g=e\equiv 1_{\cH}$ by the Reeh-Schlieder property.

If $\cB\subset \cA$ is a M\"obius covariant subnet, we define the subgroup $\Aut(\cA, \cB) < \Aut(\cA)$ by
\begin{equation*}
\Aut(\cA, \cB) = \{g\in \Aut(\cA) : g\cB(I)g^{-1} = \cB(I) \ \text{for all} \ I\in \cI \}.
\end{equation*}
Note that every $g\in \Aut(\cA, \cB)$ preserves $\cH_\cB = \bigvee_I \cB(I)\Omega$ and its restriction on $\cH_\cB$ gives an element of $\Aut(\cB)$.

Let $G < \Aut(\cA)$ be a subgroup. Then the map
\[
\cI\ni I \mapsto \cA^G(I)\subset \cA(I)
\]
gives a M\"obius covariant subnet where $\cA^G(I) = \{x\in \cA(I) \mid gxg^{-1} = x \ \text{for all} \ g\in G \}$. We will denote this subnet by $\cA^G$. If $G$ is finite, $\cA^G$ is called the \emph{orbifold} of $\cA$ by $G$. If $\cA$ is completely rational and $G$ is finite, then $\cA^G$ is also completely rational and $\mu_{\cA^G} = |G|^2\mu_\cA$~\cite[Theorem 2.6]{Xu00}.

\subsection{Twisted representations}

We will review the notion of $g$-localized endomorphisms (or $g$-twisted representations) and the braided $G$-crossed category $\GLoc \cA$ introduced by M\"uger~\cite{Mueg05}. We also give proofs of some statements needed in the later analysis.

To consider the notion of $g$-localized endomorphisms, we will work on the real line $\RR\cong S^1\setminus \{-1\}$ instead of the whole unit circle $S^1$. Let us denote $\cJ$ the set of \emph{proper intervals}(or simply \emph{intervals}) of $\RR$, i.e. non-empty bounded connected open subsets of $\RR$. For $I, J\in\cJ$, we write $I<J$ (resp. $I>J$) if $I\subset (-\infty, \inf J)$ (resp. $I\subset (\sup J, +\infty)$). We write $I^\perp = \RR\setminus \overline{I}$.
We will view $\cJ$ as a subset of $\cI$ under the identification $\RR\cong S^1\setminus \{-1\}$.

Let $\cA$ be a strongly additive local M\"obius covariant net on $S^1$. By the strongly additive assumption, $\cA$ satisfies \emph{the Haag duality on $\RR$}: We have 
\begin{equation} \label{eq:Haag_duality_on_R}
\cA(I) = \cA(I^\perp)' =\left(\bigvee_{J\in \cJ, J\subset I^\perp} \cA(J)\right)'
\end{equation}
for any $I\in \cJ$.
We denote by $\cA_\infty$ the $*$-algebra defined by
\[
    \cA_\infty := \bigcup_{I\in\cJ} \cA(I) \ .
\]
Note that we do not take any closures here. In the following, we will view $\End(\cA_\infty)$ as a $\CC$-linear monoidal category as explained in Remark~\ref{remark:non_closed_alg_tensor_cat_of_endomorphisms}.

Let $g\in \Aut(\cA)$. We denote by $\beta_g$ the corresponding automorphism on each local algebra $\cA(I)$. An endomorphism $\rho\in\End(\cA_\infty)$ is called \emph{$g$-localized in $I$} if it satisfies the following:
\begin{alignat*}{2}
\rho(x) &= \ \ x & \quad & \textrm{for all} \ J < I \ \textrm{and} \  x\in\cA(J), \\
\rho(x) &= \beta_g(x) & & \textrm{for all} \ J > I \ \textrm{and} \ x\in\cA(J).
\end{alignat*}
If $\rho$ is $g$-localized in $I$ and $J\supset I$, then $\rho$ is also $g$-localized in $J$. An endomorphism $\rho\in\End(\cA_\infty)$ is called \emph{$g$-localized} if it is $g$-localized in some interval. A $g$-localized endomorphism $\rho$ is called \emph{transportable} if for every $J\in \cJ$ there exists a unitary $U\in \cA_\infty$ such that $\tilde{\rho} = \Ad (U)\circ\rho$ is $g$-localized in $J\in \cJ$. 
For simplicity of notations, we shall denote by $\Delta_{\cA}^g(I)$ (resp. $\Delta_{\cA}^g$) the full subcategory of $\End(\cA_\infty)$ consisting of those $\rho$ that are $g$-localized in $I$ (resp. $g$-localized) and transportable.

\begin{remark}
\begin{enumerate}
\item In~\cite{Mueg05}, the notion of a $g$-localized endomorphism is defined for a net of factors on $\RR$ with slightly general assumptions. In this article, we only consider a case that the underlying net is obtained from a strongly additive local M\"obius covariant net.
\item If $g = e \equiv 1_\cH$, the category $\Delta_{\cA}^e$ is the same as the category of DHR endomorphisms of $\cA$~\cite{DHRI71,DHRII74}. They are equivalent to the category of representations of $\cA$ (cf.~\cite[Theorem 2.31]{Mueg05}) and so an element of $\Delta_{\cA}^e$ corresponds to a representation of $\cA$.
For a general element $g\in \Aut(\cA)$, a $g$-localized endomorphism is also called a \emph{$g$-twisted representation}.
\end{enumerate}
\end{remark}

If $\rho\in \Delta_{\cA}^g(I)$, then the Haag duality on $\RR$ implies $\rho(\cA(I))\subset \cA(I)$~\cite[Lemma 2.12]{Mueg05} and one can consider $\rho$ as an element of $\End(\cA(I))$. In general, if $\rho, \sigma\in \End(\cA_\infty)$ preserve $\cA(I)$ for some $I\in \cJ$, one can consider $\rho$ and $\sigma$ as elements of $\End(\cA(I))$ and consider the intertwiner space between them in $\cA(I)$. In such a situation, we will write
\begin{equation*}
\Hom_{\cA(I)}(\rho, \sigma) :=  \{t\in \cA(I) \mid t\rho(x) = \sigma(x)t \ \text{for all} \ x\in \cA(I)\}.
\end{equation*}
An element of $\Hom_{\cA(I)}(\rho, \sigma)$ is called a \emph{local intertwiner} from $\rho$ to $\sigma$. When we consider the intertwiner space in $\cA_\infty$, we also write
\begin{equation*}
\Hom_{\cA_\infty}(\rho, \sigma) := \Hom(\rho, \sigma) \equiv \{t\in \cA_\infty \mid t\rho(x) = \sigma(x)t \ \text{for all} \ x\in \cA_\infty\}
\end{equation*}
if we want to emphasize the underlying algebra.
An element of $\Hom_{\cA_\infty}(\rho, \sigma)$ is called a \emph{global intertwiner}. In general, the spaces of local and global intertwiners do not coincides.

\begin{lemma} \label{lemma_intertwiners_GLoc}
Let $g, h\in \Aut(\cA)$. For $\rho\in \Delta_{\cA}^g(I)$ and $\sigma\in \Delta_{\cA}^h(I)$, the following hold:
\begin{enumerate}
\item If $g=h$, then we have
    \begin{equation*}
    \Hom_{\cA_\infty}(\rho, \sigma) = \Hom_{\cA(I)}(\rho, \sigma)\subset \cA(I).
    \end{equation*}
\item If $g\ne h$, then there is no unitary in $\Hom_{\cA_\infty}(\rho, \sigma)$.
\item If the closed subgroup generated by $g$ and $h$ is compact, then we have 
    \begin{equation*}
        \Hom_{\cA_\infty}(\rho, \sigma) \subset \cA(I).
    \end{equation*}
\end{enumerate}

\end{lemma}
\begin{proof}

(i) It was proved that $\Hom_{\cA_\infty}(\rho, \sigma)\subset \cA(I)$ holds in~\cite[Lemma 2.13]{Mueg05}. Hence $\Hom_{\cA_\infty}(\rho, \sigma)\subset \Hom_{\cA(I)}(\rho, \sigma)$ follows by restricting $\rho$ and $\sigma$ to the local algebra $\cA(I)$. As in the proof of \cite[Lemma 2.4]{BE1}, one can prove $\Hom_{\cA_\infty}(\rho, \sigma)\supset \Hom_{\cA(I)}(\rho, \sigma)$ using the strong additivity of $\cA$.

\smallskip

(ii) The statement follows from~\cite[Remark 2.7]{Mueg05}.

\smallskip

(iii) Let $t\in \Hom_{\cA_\infty}(\rho, \sigma)$. Since $\rho(x) = \sigma(x) = x$ for $x\in \cA(I_-)$ with $I_- < I$, by the Haag duality we have $t\in \cA(J)$ for some interval $J$ such that $\inf J = \inf I$. If $J\subset I$, we get $t\in \cA(I)$ which is the desired conclusion. In the following, we assume $J\supset I$ and set $J_1 = J\setminus \overline{I}$. Hence $I$ and $J_1$ are adjacent interval obtained by removing a single point from $J$. Let $\cB = \cA^{\overline{\langle g, h\rangle}}$ be the fixed point net by the closed subgroup $\overline{\langle g, h\rangle}$ generated by $g$ and $h$, which is a compact group by our assumption. For $x\in \cB(J_1)$, we have $tx = t\beta_g(x) = \beta_h(x)t = xt$. Thus we get $t\in \cA(J)\cap \cB(J_1)'$. Since $\cB\subset \cA$ is a strongly additive pair in the sense of \cite{Xu05}, we have $\cA(J)\cap \cB(J_1)' = \cA(I)$ and hence $t\in \cA(I)$.
\end{proof}

\begin{remark}
If $\cA$ has the split property, then the group $\Aut(\cA)$ is compact~\cite{DL84} and the condition in (iii) always holds. It is known that the split property is automatic for a conformal net~\cite{MTW18}.
\end{remark}

We now turn to the definition of the category $\GLoc \cA$. In the following, let us fix a compact subgroup $G < \Aut(\cA)$.
\begin{definition}
The category $\GLoc \cA$ is defined to be the full subcategory of $\End(\cA_\infty)$ whose objects are finite direct sums of $G$-localized transportable endomorphisms of $\End(\cA_\infty)$, i.e. $\rho\in \GLoc \cA$ if and only if there exist $g_i\in G$, $\rho_i\in \Delta_{\cA}^{g_i}$, and isometries $s_i\in \cA_\infty$ satisfying $\sum_i s_i s_i^* = 1$ such that
\begin{equation*}
\rho(\cdot) = \sum_i s_i\rho_i(\cdot)s_i^* \ .
\end{equation*} 
\end{definition}

The category $\GLoc \cA$ is a $\CC$-linear category with simple unit, finite direct sums and subobjects. By (i) and (iii) of Lemma~\ref{lemma_intertwiners_GLoc}, $\GLoc\cA$ is in fact a C*-tensor category.
It has a structure of strict $G$-crossed category in the sense of Definition~\ref{def_braided_G_crossed_category}~\cite[Proposition 2.10]{Mueg05}: By (ii) of Lemma~\ref{lemma_intertwiners_GLoc} and definition of $\GLoc \cA$, $(\GLoc \cA)_g = \Delta_{\cA}^g$ gives a $G$-graiding on $\GLoc \cA$. The strict action of $G$ is given by
\begin{alignat*}{2}
\gamma_g(\rho) &= \beta_g\circ\rho\circ\beta_g^{-1} &\qquad &\text{for} \ \rho\in\End(N), \\
\gamma_g(s) &= \beta_g(s) & &\text{for} \ s\in \Hom(\rho, \sigma).
\end{alignat*}
Moreover, $\GLoc \cA$ has a $G$-crossed braiding and becomes a braided $G$-crossed category~\cite[Proposition 2.17]{Mueg05}. For the later use, we explain its construction below. We first recall the following key lemma~\cite[Lemma 2.14]{Mueg05}.
\begin{lemma} \label{lemma:twisted_commutativity_of_g_loc_endo}
Let $I_1, I_2\in \cJ$ be intervals such that $I_1 < I_2$. For every $g, h\in G$, $\rho_1\in \Delta_{\cA}^g(I_1)$ and $\rho_2\in \Delta_{\cA}^h(I_2)$, we have
\begin{equation*}
    \rho_1\otimes\rho_2 = \gamma_g(\rho_2)\otimes \rho_1 .
\end{equation*}
\end{lemma}
Let $\rho\in \Delta_{\cA}^g(I)$ and $\sigma$ be $G$-localized in $J$. Choose an interval $I_- < J$. By the transportability, there is a unitary $U_{\rho;I, I_-}\in \cA_\infty$ such that $\widetilde{\rho} = \Ad (U_{\rho;I, I_-})\circ \rho\in \Delta_{\cA}^g(I_-)$. Note that we have $\widetilde{\rho}\otimes \sigma = \gamma_g(\sigma)\otimes \widetilde{\rho}$ by the above lemma. Then the braiding operator $c(\rho, \sigma)$ is defined by the composite
\begin{equation} \label{eq:def_of_GLoc_braiding}
\xymatrix{
c(\rho, \sigma)\colon \rho\otimes \sigma \ar[rr]^-{U_{\rho;I, I_-}\otimes 1_\sigma} && \widetilde{\rho}\otimes \sigma = \gamma_g(\sigma)\otimes \widetilde{\rho} \ar[rr]^-{1_{\gamma_g(\sigma)}\otimes U_{\rho;I, I_-}^*} && \gamma_g(\sigma)\otimes \rho
} \ ,
\end{equation}
namely we define $c(\rho, \sigma)=\gamma_g(\sigma)(U_{\rho;I, I_-}^*)U_{\rho;I, I_-}$.
Note that $c(\rho, \sigma)$ is a unitary by its construction. It was showed~\cite[Proposition 2.17]{Mueg05} that $c(\rho, \sigma)$ does not depend on the choice of an interval $I_-$ and a unitary $U_{\rho;I, I_-}\in \cA_\infty$ as long as they satisfy conditions stated above.
\begin{remark}
On the full subcategory $\Delta_{\cA}^e$, the $G$-crossed braiding $c(\rho, \sigma)$ coincides with the ordinary DHR braiding $\eps(\rho, \sigma)$ by its construction.
\end{remark}

We give another formula for the $G$-crossed braiding for later applications.
\begin{lemma} \label{lemma:GLoc_braiding_another_formula}
Let $\rho\in \Delta_{\cA}^g(I)$ and $\sigma\in \Delta_{\cA}^h(I)$. Suppose that $I_+\in \cJ$ is an interval with $I_+ > I$ and $U_{\sigma;I, I_+}\in \cA_\infty$ is a unitary such that $\widetilde{\sigma} = \Ad(U_{\sigma:I, I_+})\circ \sigma\in \Delta_{\cA}^h(I_+)$. Then the composite
\begin{equation*}
\xymatrix{
    \gamma_g(U_{\sigma;I, I_+}^*)\rho(U_{\sigma;I, I_+})\colon \rho\otimes \sigma \ar[rr]^-{1_\rho\otimes U_{\sigma;I, I_+}} && \rho\otimes \widetilde{\sigma} = \gamma_g(\widetilde{\sigma})\otimes \rho \ar[rr]^-{\gamma_g(U_{\sigma;I, I_{+}})\otimes 1_\rho} && \gamma_g(\sigma)\otimes \rho
    }
\end{equation*}
is equal to $c(\rho, \sigma)$.
\end{lemma}

\begin{proof}
Taking an interval $I_- < I$ and a unitary $U_{\rho;I, I_-}$ such that $\widetilde{\rho} = \Ad(U_{\rho;I, I_-})\circ \rho\in \Delta_{\cA}^g(I_-)$, we have $c(\rho, \sigma) = \gamma_g(\sigma)(U_{\rho;I, I_-}^*)U_{\rho;I, I_-}$ as explained above. To simplify the notations, we write $U_1 = U_{\rho;I, I_-}$ and $U_2 = U_{\sigma;I, I_+}$.
Since $I_- < I < I_+$, we have $\widetilde{\rho}(
U_2^*) = \beta_g(U_2^*)=\gamma_g(U_2^*)$ and $\gamma_g(\widetilde{\sigma})(U_1^*) = U_1^*$.
Using these two identities and the intertwining properties of $U_1$ and $U_2$, we compute
\begin{align*}
c(\rho, \sigma)\cdot \left(\gamma_g(U_{\sigma;I, I_+}^*)\rho(U_{\sigma;I, I_+})\right)^* &= \gamma_g(\sigma)(U_1^*)U_1\rho(U_2^*)\gamma_g(U_2) \\
&= \gamma_g(\sigma)(U_1^*)\widetilde{\rho}(U_2^*)U_1\gamma_g(U_2) \\
&= \gamma_g(\sigma)(U_1^*)\gamma_g(U_2^*)U_1\gamma_g(U_2) \\
&= \gamma_g(U_2^*)\gamma_g(\widetilde{\sigma})(U_1^*)U_1 \gamma_g(U_2) \\
&= \gamma_g(U_2^*)U_1^*U_1 \gamma_g(U_2) = 1.
\end{align*}
Thus we get the equality $c(\rho, \sigma) = \gamma_g(U_{\sigma;I, I_+}^*)\rho(U_{\sigma;I, I_+})$, which completes the proof.
\end{proof}

\section{$\alpha$-induction for twisted representations} \label{section:alpha_induction_for_twisted_reps}

This section is a main part of this article.

\subsection{Assumptions and notations} \label{subsection:ind_for_twisted_reps_assumptions}

Suppose that we have a given local M\"obius covariant net $\cA$ and its M\"obius covariant subnet $\cB\subset \cA$. We assume that both $\cA$ and $\cB$ are strongly additive and the index $[\cA:\cB]$ is finite. We also assume that we have a given compact subgroup $G < \Aut(\cA, \cB)$. We denote by $G'$ the subgroup of $\Aut(\cB)$ obtained by restricting each element of $G$ to $\cH_\cB$. We write $g'$ the image of $g\in G$ under the canonical surjection $G\to G'$.
We use the following notations for the group actions on these nets and corresponding categories of twisted representations: We denote by
\begin{itemize}
\item $\tilbeta_g$ the action of $g\in G$ on each local algebra $\cA(I)$ of $\cA$.
\item $\tilgamma_g$ the action of $g\in G$ on the category $\GLoc \cA$.
\item $\beta_{g'}$ the action of $g'\in G'$ on each local algebras $\cB(I)$ of $\cB$.
\item $\gamma_{g'}$ the action of $g'\in G'$ on the category $\GpLoc \cB$
\end{itemize}
Since the action $\tilgamma_g$ of $g\in G$ on $\GLoc \cA$ naturally extends to $\End(\cA_\infty)$, we also denote by $\tilgamma_g$ this extension. We will write $c^+(\lambda, \mu)\equiv c(\lambda, \mu)$ for the $G'$-crossed braiding of $\GpLoc \cB$ and $c^-(\lambda, \mu)$ for its opposite as explained in Section~\ref{section:braided_G_crossed_category}.

\subsection{Definition of induction} \label{subsection:ind_for_twisted_reps_def}

Let $\iota\colon \cB_\infty \hookrightarrow \cA_\infty$ and $\iota_I\colon \cB(I) \hookrightarrow \cA(I)$ be the inclusion maps. Since restriction of $\cB\subset \cA$ to $\cJ$ is a finite index standard net of subfactors in the sense of~\cite{LR95}, the results in the same article implies the following.

\begin{proposition} \label{prop:standard_net_of_subfactors_and_conjugate}
Let $\cB\subset \cA$ as above. For every $I\in \cJ$, there exists a homomorphism $\iotabar\colon \cA_\infty\to \cB_\infty$, $w\in \cB(I)$ and $v\in \cA(I)$ satisfying the following:
\begin{enumerate}
\item For every $J\in\cJ$ with $J\supset I$, $\iotabar|_{\cA(J)}$ is a conjugate of $\iota_J$ and $(w, v)$ is a standard solution of conjugate equation for $\iota_J$ and $\iotabar|_{\cA(J)}$.
\item $\iotabar$ acts identically on $\cB(I)'\cap \cA_\infty$.
\end{enumerate}
If $\widetilde{I}\in \cJ$ is another choice of an interval and $\iotabar'$ and $(w', v')$ are the corresponding homomorphism and operators, then there exists a unitary in $u\in \cB_\infty$ satisfying
\begin{gather*}
u\in \Hom(\iotabar, \iotabar'), \\
w' = uw, \quad v' = \iota(u)v.
\end{gather*}
\end{proposition}

In the following, let us fix an arbitrary interval $I_0\in \cJ$ and a homomorphism $\iotabar\colon \cA_\infty\to \cB_\infty$ as in Proposition~\ref{prop:standard_net_of_subfactors_and_conjugate}. Note that we have $\theta = \iotabar\iota \in \Delta_{\cB}^e(I_0)$. We denote by $(w, v)$ the corresponding standard solution.
In the following, we will drop subscripts for $\iota$ if no confusion arise. We denote by $\Theta = (\theta=\iotabar\iota, w, x=\iotabar(v))$ a Q-system associated with the subfactor $\cB(I_0)\subset \cA(I_0)$. We note that the same formula defines a Q-system associated with $\cB(I)\subset \cA(I)$ for every $I\in\cJ$ with $I\supset I_0$.
We denote by $z=\{z_g\colon \gamma_{g'}(\theta)\to \theta \}_{g\in G}$ the $G$-equivariant structure of $\Theta$ associated with the action of $G$ on $\cB(I_0)\subset \cA(I_0)$ as in Lemma~\ref{lemma:G_equivariant_str_on_Q_system}.
For the convenience, we collect relevant properties under notations of this section:
\begin{gather}
\tilbeta_g\circ \iota = \iota\circ \beta_{g'}\quad \text{for all} \ g\in G, \label{eq:beta_comm_iota} \\
z_g\in \cB(I_0) \ \text{and} \ \tilbeta_g(v) = z_g^* v \quad \text{for all} \ g\in G, \label{eq:characterization_z_g} \\
z_g\in \Hom(\beta_{g'}\iotabar, \iotabar\tilbeta_g) \quad \text{for all} \ g\in G, \label{eq:z_g_intertwining_property}\\
z_{gh} = z_g\gamma_{g'}(z_g) \equiv z_g\beta_{g'}(z_h) \quad \text{for all} \ g,h\in G, \label{eq:z_g_cocycle_eq}\\
\theta(z_g)z_g\gamma_{g'}(x) \equiv \theta(z_g)z_g\beta_{g'}(\iotabar(v)) = \iotabar(v)z_g \equiv x z_g \quad \text{for all} \ g\in G. \label{eq:z_g_alg_isom}
\end{gather}
where the last identity comes from the fact that $z_g$ is a unitary isomorphism of Q-system from $g(\Theta)$ to $\Theta$.

We will define two kind of induced endomorphisms for twisted representations using the $G$-crossed braiding and $G$-equivariant structure, which generalize $\alpha$-induced endomorphisms.
To give their definition, we prove the following lemma as in the case of usual $\alpha$-induction.

\begin{lemma} \label{lemma:basic_lemma_for_def_galpha}
For $g\in G$ and $\lambda\in \Delta_{\cB}^{g'}(I_0)$, we have
\begin{align*}
\Ad(z_g c^+(\lambda, \theta))\circ\lambda\circ\overline{\iota}(v) &= \theta(c^+(\lambda, \theta)^* z_g^*)\overline{\iota}(v), \\
\Ad(c^-(\lambda, \theta))\circ\lambda\circ\overline{\iota}(v) &= \theta(c^-(\lambda, \theta)^*)\overline{\iota}(v).
\end{align*}
\end{lemma}

\begin{proof}
For the first identity, it suffices to show that
\begin{equation*}
\theta(z_g c^+(\lambda, \theta))z_g c^+(\lambda, \theta)\lambda(\iotabar(v)) = \iotabar(v)z_g c^+(\lambda, \theta).
\end{equation*}
Using the properties of the braiding Eq.~\eqref{naturality_gcbraiding} and Eq.~\eqref{braid_rel_gcbraiding}, and the property of $z_g$ Eq.~\eqref{eq:z_g_alg_isom}, we compute
\begin{align*}
\theta(z_g c^+(\lambda, \theta))\cdot z_g c^+(\lambda, \theta)\lambda(\iotabar(v)) &= \theta(z_g)z_g \cdot \gamma_g(\theta)(c^+(\lambda, \theta))\cdot c^+(\lambda, \theta)\lambda(\iotabar(v)) \\
&= \theta(z_g)z_g c^+(\lambda, \theta^2)\lambda(\iotabar(v)) \\
&= \theta(z_g)z_g\gamma_{g'}(\iotabar(v))c^+(\lambda, \theta) \\
&= \theta(z_g)z_g\beta_{g'}(\iotabar(v))c^+(\lambda, \theta) \\
&= \iotabar(v) z_gc^+(\lambda, \theta)
\end{align*}
and we get the desired identity. The second identity is obtained similarly by using only the properties of the opposite braiding Eq.~\eqref{naturality_gcopbraiding} and Eq.~\eqref{braid_rel_gcopbraiding}.
\end{proof}
Let $I\supset I_0$ be an interval. For $n\in \cB(I)$, we have 
\begin{subequations} \label{eq:basic_identity_on_sub}
\begin{align} 
\Ad(z_g c^+(\lambda, \theta))\circ\lambda\circ\overline{\iota}(\iota(n)) &= \Ad(z_g c^+(\lambda, \theta))\circ\lambda\circ\theta(n) = \theta\circ\lambda(n), \\
\Ad(c^-(\lambda, \theta))\circ\lambda\circ\overline{\iota}(\iota(n)) &= \Ad(c^-(\lambda, \theta))\circ\lambda\circ\theta(n) = \theta\circ\lambda(n).
\end{align}
\end{subequations}
where we used $z_g c^+(\lambda, \theta), c^-(\lambda, \theta)\in \Hom(\lambda\theta, \theta\lambda)$.
Since $\cA(I) = \cB(I)v$, we have the following corollary.

\begin{corollary} \label{cor:basic_cor_for_def_galpha}
For $g\in G$ and $I\supset I_0$ and $\lambda\in \Delta_{\cB}^{g'}(I_0)$, we have
\begin{align*}
\Ad(z_g c^+(\lambda, \theta))\circ\lambda\circ\overline{\iota}(\cA(I)) &\subset \iotabar(\cA(I)), \\
\Ad(c^-(\lambda, \theta))\circ\lambda\circ\overline{\iota}(\cA(I)) &\subset \iotabar(\cA(I)).
\end{align*}
\end{corollary}

We now give the definition of induced endomorphisms of twisted representation.

\begin{definition}
For $g\in G$ and $\lambda\in \Delta_{\cB}^{g'}(I_0)$, we define the two endomorphisms of $\End(\cA_\infty)$ by
\begin{align*}
\gpalpha{g}{\lambda} &= \iotabar^{-1}\circ \Ad(z_g c^+(\lambda, \theta))\circ\lambda\circ \iotabar, \\
\gmalpha{\lambda} &= \iotabar^{-1}\circ \Ad(c^-(\lambda, \theta))\circ\lambda\circ \iotabar.
\end{align*}
\end{definition}

\begin{remark}
\begin{enumerate}
\item By Corollary~\ref{cor:basic_cor_for_def_galpha}, $\gpalpha{g}{\lambda}$ and $\gmalpha{\lambda}$ preserve $\cA(I)$ if $I\supset I_0$ and one can consider them as elements of $\End(\cA(I))$.
\item If $g=e$, then the endomorphisms as above are equal to usual $\alpha$-induced endomorphisms. If $g\ne e$ but $g' = e$, then the above $\gmalpha{\lambda}$ is still equal to the usual $\alpha^-$-induced endomorphism $\malpha{\lambda}$. The endomorphism $\gmalpha{\lambda}$ is defined by using only the opposite braiding, but $\gpalpha{g}{\lambda}$ is defined by using the braiding and the $G$-equivariant structure.
\end{enumerate}
\end{remark}

\subsection{Basic properties of induced endomorphisms} \label{subsection:ind_for_twisted_reps_basic_properties}

In this subsection, we will see some basic properties of induced endomorphisms. By Eq.~\eqref{eq:basic_identity_on_sub}, $\gpalpha{g}{\lambda}$ and $\gmalpha{\lambda}$ are extensions of $\lambda$, i.e. $\gpalpha{g}{\lambda}\circ\iota = \iota\circ\lambda$ and $\gmalpha{\lambda}\circ\iota = \iota\circ\lambda$. These two extension satisfy
\begin{equation} \label{eq:values_of_galpha_on_v}
\gpalpha{g}{\lambda}(v) = c^+(\lambda, \theta)^* z_g^*v \ , \quad \gmalpha{\lambda}(v) = c^-(\lambda, \theta)^*v \ ,
\end{equation}
by Lemma~\ref{lemma:basic_lemma_for_def_galpha}. 
Since we have $\cA(I) = \cB(I)v$ for $I\supset I_0$, these properties characterize the induced endomorphisms.

In the following, we will derive some formulas using the properties of the braiding and the $G$-equivariant structure $z$.
We first see the multiplicativity.

\begin{proposition}
Let $g, h\in G$. For every $\lambda\in \Delta_{\cB}^{g'}(I_0)$ and $\mu \in \Delta_{\cB}^{h'}(I_0)$, we have
\begin{align*}
\gpalpha{g}{\lambda}\circ \gpalpha{h}{\mu} &= \gpalpha{gh}{\lambda\mu} \ , \\
\gmalpha{\lambda}\circ \gmalpha{\mu} &= \gmalpha{\lambda\mu} \ ,
\end{align*}
in $\End(\cA_\infty)$.
\end{proposition}

\begin{proof}
The two formulas can be checked by direct computation.
For the first formula, we compute
\begin{align*}
\gpalpha{g}{\lambda}\circ \gpalpha{h}{\mu} &=
\iotabar^{-1}\circ \Ad(z_g c^+(\lambda, \theta))\circ \lambda \circ \Ad(z_h c^+(\mu, \theta))\circ \mu \iotabar \\
&= \iotabar^{-1}\circ \Ad(z_g c^+(\lambda, \theta)\lambda(z_h) \lambda(c^+(\mu, \theta)))\circ \lambda\mu \circ \iotabar \\
&= \iotabar^{-1}\circ \Ad(z_g \gamma_{g'}(z_h) c^+(\lambda, \gamma_h(\theta))\lambda(c^+(\mu, \theta)))\circ \lambda\mu \circ \iotabar \\
&= \iotabar^{-1}\circ \Ad(z_{gh} c^+(\lambda\mu, \theta))\circ \lambda\mu \circ \iotabar \\
&= \gpalpha{gh}{\lambda\mu},
\end{align*}
where we used the properties of braiding Eq.~\eqref{naturality_gcbraiding} and Eq.~\eqref{braid_rel_gcbraiding}, and Eq.~\eqref{eq:z_g_cocycle_eq}.
The second formula can be checked in a similar way by using the properties of the opposite braiding Eq.~\eqref{naturality_gcopbraiding} and Eq.~\eqref{braid_rel_gcopbraiding}.
\end{proof}

We next see the covariance property of induced endomorphisms.
\begin{proposition} \label{prop:covariance_galpha}
Let $g, k\in G$ and $\lambda\in \Delta_{\cB}^{g'}(I_0)$. Then we have 
\begin{align*}
\tilgamma_k(\gpalpha{g}{\lambda}) &= \gpalpha{kgk^{-1}}{\gamma_{k'}(\lambda)} \ , \\
\tilgamma_k(\gmalpha{\lambda}) &= \gmalpha{\gamma_{k'}(\lambda)} \ ,
\end{align*}
in $\End(\cA_\infty)$.
\end{proposition}

\begin{proof}
Let $I\in \cJ$ such that $I\supset I_0$. We will see that the two formulas hold on $\cA(I)$. Since $\gpalpha{g}{\lambda}$ and $\gmalpha{\lambda}$ act as $\lambda$ on $\cB(I)$, it is easily seen that the desired formulas hold on $\cB(I)$. 
Thus it suffices to see that $\tilgamma_k(\gpalpha{g}{\lambda})(v)= \gpalpha{kgk^{-1}}{\gamma_{k'}(\lambda)}(v)$ and $\tilgamma_k(\gmalpha{\lambda})(v) = \gmalpha{\gamma_{k'}(\lambda)}(v)$, because $\cA(I) = \cB(I)v$. 

We first see the equality for the first one. The left-hand side is computed as
\begin{align*}
\tilgamma_k(\gpalpha{g}{\lambda})(v) &= \tilbeta_k\circ \gpalpha{g}{\lambda}\circ\tilbeta_k^{-1}(v) \\
&= \tilbeta_k\circ \gpalpha{g}{\lambda}(z_{k^{-1}}^* v) \\
&= \tilbeta_k(\lambda(z_{k^{-1}}^*)c^+(\lambda, \theta)^* z_g^*v) \\
&= \tilbeta_k(c^+(\lambda, \gamma_{{k'}^{-1}}(\theta))^* \beta_{g'}(z_{k^{-1}}^*) z_g^*v) \\
&= c^+(\gamma_{k'}(\lambda), \theta)^* \beta_{k'g'}(z_{k^{-1}}^*)\beta_{k'}(z_g^*)z_k^* v \ ,
\end{align*}
where we used the definition of $z$, Eq.~\eqref{eq:values_of_galpha_on_v}, the properties of braiding Eq.~\eqref{naturality_gcbraiding} and Eq.~\eqref{covariance_gcbraiding}. On the other hand, the right-hand side is
\begin{equation*}
\gpalpha{kgk^{-1}}{\gamma_{k'}(\lambda)}(v) = c^+(\gamma_{k'}(\lambda), \theta)^* z_{kgk^{-1}}^* v \ .
\end{equation*}
Since $z_{kgk^{-1}} = z_k\beta_{k'}(z_{gk^{-1}}) = z_k\beta_{k'}(z_g\beta_{g'}(z_{{k'}^{-1}})) = z_k\beta_{k'}(z_g)\beta_{k'g'}(z_{{k'}^{-1}})$ by using Eq.~\eqref{eq:z_g_cocycle_eq} repeatedly, we get the desired equality.

We next see the equality for the second one. The left-hand side is computed as:
\begin{align*}
\tilgamma_k(\gmalpha{\lambda})(v) &= \tilbeta_k\circ \gmalpha{\lambda}\circ \tilbeta_k^{-1}(v) \\
&= \tilbeta_k\circ \gmalpha{\lambda}(z_{k^{-1}}^*v) \\
&= \tilbeta_k(\lambda(z_{k^{-1}}^*)c^-(\lambda, \theta)^* v) \\
&= \tilbeta_k(c^-(\lambda, \gamma_{{k'}^{-1}}(\theta))^* z_{k^{-1}}^* v) \\
&= c^-(\gamma_{k'}(\lambda), \theta)^* \beta_{k'}(z_{k^{-1}}^*) z_k^* v \\
&= c^-(\gamma_{k'}(\lambda), \theta)^* v,
\end{align*}
where we used Eq.~\eqref{eq:characterization_z_g}, Eq.~\eqref{eq:values_of_galpha_on_v}, the properties of opposite braiding Eq.~\eqref{naturality_gcopbraiding} and Eq.~\eqref{covariance_gcopbraiding}. Moreover, 
we used $1 = z_{k' {k'}^{-1}} = z_{k'}\beta_{k'}(z_{{k'^{-1}}})$ in the last line. On the other hand, the right-hand side is $\gmalpha{\tilgamma_{k'}(\lambda)}(v) = c^-(\tilgamma_{k'}(\lambda), \theta)^*v$ by Eq.~\eqref{eq:values_of_galpha_on_v} and we also get the desired equality.
\end{proof}

We then see the localization properties of induced endomorphisms. The following proposition says that $\gpalpha{g}{\lambda}$ is $g$-localized in a left half-line and $\gmalpha{\lambda}$ is localized in a right half-line.

\begin{proposition}~\label{prop:galpha_localizable}
Let $\lambda\in \Delta_{\cB}^{g'}(I_0)$ and $I\in \cJ$.
\begin{enumerate}
    \item If $I > I_0$, then we have $\gpalpha{g}{\lambda}(m) = \tilbeta_g(m)$ for all $m\in \cA(I)$.
    \item If $I < I_0$, then we have $\gmalpha{\lambda}(m) = m$ for all $m\in \cA(I)$.
\end{enumerate}
\end{proposition}

\begin{proof}
(i) We use an explicit form of $c^+(\lambda, \theta)$ as in Eq.~\eqref{eq:def_of_GLoc_braiding}. Pick an interval $I_-\in\cJ$ with $I_- < I_0$ and a unitary $U\in \cB_\infty$ such that $\tillambda = \Ad(U)\circ \lambda\in \Delta_{\cB}^{g'}(I_-)$. Then $c^+(\lambda, \theta)$ is given by $c^+(\lambda, \theta) = \gamma_{g'}(\theta)(U^*)\cdot U$. We note that $U\in A(\widetilde{J})$ as long as $\widetilde{J} \supset I_-\cup I_0$.

Let $I > I_0$ and $m\in \cA(I)$. Pick an interval $J\in \cJ$ such that $I\cup I_0 \subset J$ and $I_- < J$. Then we have $\iotabar(m)\in \cA(J)$. Since $\tillambda$ is $g$-localized in $I_-$, we have $\tillambda\circ \iotabar(m) = \beta_{g'}\circ\iotabar(m)$. Hence we compute
\begin{align*}
\gpalpha{g}{\lambda}(m) &= \iotabar^{-1}\circ \Ad(z_g c^+(\lambda, \theta))\circ\lambda\circ \iotabar(m) \\
&= \iotabar^{-1}\circ \Ad(z_g \cdot \gamma_{g'}(\theta)(U^*)\cdot U) \circ \lambda\circ \iotabar(m) \\
&= \iotabar^{-1}\circ \Ad(z_g\cdot \gamma_{g'}(\theta)(U^*))\circ \tillambda\circ \iotabar(m)\\
&= \iotabar^{-1}\circ \Ad(z_g\cdot \gamma_{g'}(\theta)(U^*))\circ \beta_{g'}\circ \iotabar(m) \\
&= \iotabar^{-1}\circ \Ad(\theta(U^*)z_g)\circ \beta_{g'}\circ \iotabar(m) \\
&= \iotabar^{-1}\circ \Ad(\theta(U^*))\circ\iotabar\circ \tilbeta_g(m)\\
&= \Ad(\iota(U^*))\circ \tilbeta_g(m) \ .
\end{align*}
Since $\iota(U^*)$ and $\tilbeta_g(m)\in \cA(I)$ commute by the locality of $\cA$, we get $\gpalpha{g}{\lambda}(m) = \tilbeta_g(m)$.
\smallskip

(ii) As in the proof of (i), we use an explicit form of $c^-(\lambda, \theta) = c^+(\theta, \lambda)^*$. By Lemma~\ref{lemma:GLoc_braiding_another_formula}, $c^+(\theta, \lambda)$ is given as follows. Taking an interval $I_+\in \cJ$ with $I_+ > I_0$ and a unitary $u\in \cB_\infty$ such that $\tillambda = \Ad(u)\circ \lambda\in \Delta_{\cB}^{g'}(I_+)$, we have $c^+(\theta, \lambda) = U^*\theta(U)$. Thus we get $c^-(\lambda, \theta) = \theta(U^*)U$. Using this formula, one can check the statement as in the proof of (i).
\end{proof}

Moreover, two induced endomorphisms are localizable in any half-line as below. Let us fix an arbitrary interval $\widetilde{I}\in \cJ$. Starting from $\widetilde{I}$, we choose a homomorphism $\iotabar'\colon \cA_\infty\to \cB_\infty$ and $(w', v')$ as in Proposition~\ref{prop:standard_net_of_subfactors_and_conjugate}. Let $u\in \Hom(\iotabar, \iotabar')$ be a unitary as in the same proposition. Then $\Theta=(\theta'=\iotabar'\iota, w', x'=\iotabar'(v'))$ is a Q-system equivalent to $\Theta$ via unitary isomorphism $u=u\otimes 1_\iota\in \Hom(\theta, \theta')$. We denote by $z'$ the corresponding $G$-equivariant structure.

Let $g\in G$ and $\lambda\in \Delta_{\cB}^{g'}(I_0)$. By the transportability of $\lambda$, there exists a unitary $U\in \cB_\infty$ such that $\tilde{\lambda} = \Ad(U)\circ \lambda \in \Delta_{\cB}^{g'}(\widetilde{I})$. We temporarily denote by $\tilde{\alpha}_{\tilde{\lambda}}^{g;+}$ and $\tilde{\alpha}_{\tilde{\lambda}}^-$ the induced endomorphisms defined in terms of $\theta'$ and $z'$.
\begin{lemma} \label{lemma:galpha_transportability}
Under the notations as above, $\gpalpha{g}{\lambda}$ (resp. $\gmalpha{\lambda}$) and $\tilde{\alpha}_{\tilde{\lambda}}^{g;+}$ (resp. $\tilde{\alpha}_{\tilde{\lambda}}^-$) are unitarily equivalent in $\End(\cA_\infty)$.
\end{lemma}

\begin{proof}
We only show the plus case. The minus case can be checked in the similar way. By Lemma~\ref{lemma:G_equivariant_str_on_Q_system}, we get $z_g'\gamma_{g'}(u) = uz_g$ and we have
\begin{align*}
u^* z_g'c^+(\tilde{\lambda}, \theta') U\lambda(u) &= u^* u z_g \gamma_{g'}(u^*)c^+(\tilde{\lambda}, \theta') U\lambda(u) \\
&= z_g \cdot \gamma_{g'}(\theta)(U)\cdot c^+(\lambda, \theta) \\
&= \theta(U)z_gc^+(\lambda, \theta)
\end{align*}
where we used the naturality of braiding Eq.~\eqref{naturality_gcbraiding} in the second line and the intertwining property of $z_g$ in the third line.
Using this, we have
\begin{align*}
\tilde{\alpha}_{\tilde{\lambda}}^{g;+} &= {\iotabar'}^{-1}\circ \Ad(z_g' c^+(\tilde{\lambda}, \theta'))\circ \tilde{\lambda} \circ \iotabar' \\
&= \iotabar^{-1} \circ \Ad(u^* z_g'c^+(\tilde{\lambda}, \theta') U\lambda(u))\circ \lambda\circ \iotabar \\
&= \iotabar\circ \Ad(\theta(U)z_g c^+(\lambda, \theta))\circ \lambda\circ \iotabar \\
&= \Ad(\iota(U))\circ \gpalpha{g}{\lambda},
\end{align*}
which completes the proof.
\end{proof}

\subsection{Intertwiner spaces of induced endomorphisms} \label{subsection:ind_for_twisted_reps_intertwiners}
We now study the intertwiner spaces between induced endomorphisms using essentially the same strategy as~\cite{BE1}. By the locality of $\cA$, we have the following lemma~\cite{LR95} (also cf.~\cite[Lemma 3.4]{BE1}), which implies the commutativity of the Q-system $\Theta=(\theta, w, x=\iotabar(v))$.
\begin{lemma} \label{lemma:locality_of_Q_system_net_of_subfactor}
We have
\begin{align*}
c^+(\theta, \theta)\iotabar(v) &= c^-(\theta, \theta)\iotabar(v) = \iotabar(v), \\
c^+(\theta, \theta)v^2 &= c^-(\theta, \theta)v^2 = v^2.
\end{align*}
\end{lemma}

In the subsequent analysis, we temporarily view $\gpalpha{g}{\lambda}$ and $\gmalpha{\lambda}$ as elements of $\End(\cA(I_0))$ and consider local intertwiner spaces. We denote by $\Hom_{\cB(I_0), \cA(I_0)}(\iota\lambda, \iota\mu)$ the space of local intertwiners $\{t\in \cA(I_0) : t\iota\lambda(n) = \iota\mu(n)t \ \text{for all} \ n\in \cB(I_0)\}$. The following theorem generalize~\cite[Theorem 3.9]{BE1}.

\begin{theorem} \label{thm_galpha_intertwiner}
Let $g\in G$ and $\lambda, \mu\in \Delta_{\cB}^{g'}(I_0)$. Then we have
\begin{enumerate}
\item $\Hom_{\cA(I_0)}(\gpalpha{g}{\lambda}, \gpalpha{g}{\mu}) = \Hom_{\cB(I_0), \cA(I_0)}(\iota\lambda, \iota\mu)$.
\item $\Hom_{\cA(I_0)}(\gmalpha{\lambda}, \gmalpha{\mu})= \Hom_{\cB(I_0), \cA(I_0)}(\iota\lambda, \iota\mu)$.
\end{enumerate}
In particular, we have
\begin{equation*}
\langle \gpalpha{g}{\lambda}, \gpalpha{g}{\mu} \rangle_{\cA(I_0)} = \langle \theta\lambda, \mu \rangle_{\cB(I_0)} = \langle \gmalpha{\lambda}, \gmalpha{\mu} \rangle_{\cA(I_0)}.
\end{equation*}
\end{theorem}

\begin{proof}
(i) Since $\gpalpha{g}{\lambda}$ and $\gpalpha{g}{\mu}$ act as $\lambda$ and $\mu$ on $\cB(I_0)$ respectively, we have 
\[
\Hom_{\cA(I_0)}(\gpalpha{g}{\lambda}, \gpalpha{g}{\mu}) \subset \Hom_{\cB(I_0), \cA(I_0)}(\iota\lambda, \iota\mu).
\]
Let $t\in \Hom_{\cB(I_0), \cA(I_0)}(\iota\lambda, \iota\mu)$. We set $s:= \iotabar(t)\in \Hom_{\cB(I_0)}(\theta\lambda, \theta\mu)$. Since $\cA(I_0) = \cB(I_0)v$, it suffices to show that $t\gpalpha{g}{\lambda}(v) = \gpalpha{g}{\mu}(v)t$.
Applying $\iotabar$ and by Eq.~\eqref{eq:values_of_galpha_on_v}, it is equivalent to the equality $s\theta(c^+(\lambda, \theta)^* z_g^*)\iotabar(v) = \theta(c^+(\mu, \theta)^* z_g^*)\iotabar(v) s$. By the property of braiding Eq.~\eqref{naturality_gcbraiding} and Eq.~\eqref{braid_rel_gcbraiding} and the intertwining property of $z_g$, we have
\begin{align*}
c^+(\theta, \theta)\theta(z_g c^+(\lambda, \theta)) &= z_g c^+(\theta, \gamma_{g'}(\theta))\theta(c^+(\lambda, \theta)) \\
&= z_g c^+(\theta\lambda, \theta).
\end{align*}
Using this formula and Lemma \ref{lemma:locality_of_Q_system_net_of_subfactor}, we now compute
\begin{align*}
s\theta(c^+(\lambda, \theta)^* z_g^*) \iotabar(v) &= s\theta(c^+(\lambda, \theta)^*z_g^*)c^+(\theta, \theta)^*\iotabar(v) \\
&= s c^+(\theta\lambda, \theta)^* z_g^* \iotabar(v) \\
&= c^+(\theta\lambda, \theta)^* \gamma_{g'}(\theta)(s) z_g^* \iotabar(v) \\
&= c^+(\theta\lambda, \theta)^* z_g^* \theta(s)\iotabar(v) \\
&= c^+(\theta\lambda, \theta)^* z_g^* \iotabar(v) s \\
&= \theta(c^+(\mu, \theta)^* z_g^*)c^+(\theta, \theta)^* \iotabar(v) s \\
&= \theta(c^+(\mu, \theta)z_g^*)\iotabar(v) s,
\end{align*}
and we get the desired equality, which proves the statement.

\smallskip

(ii) Using the properties of opposite braiding, one can prove the statement as in the proof of~\cite[Lemma 3.5]{BE1} (or the proof of (i) above).

\smallskip

The last formulas follow from the Frobenius reciprocity, namely there is a linear bijection between $\Hom_{\cB(I_0), \cA(I_0)}(\iota\lambda, \iota\mu)$ and $\Hom_{\cB(I_0)}(\theta\lambda, \mu)$. Explicitly, two linear maps given by
\begin{subequations} \label{eq:frobenius_reciprocity_for_main_formula_galpha}
\begin{align}
\Hom_{\cB(I_0), \cA(I_0)}(\iota\lambda, \iota\mu)\ni t &\longmapsto w^*\iotabar(t) \in \Hom_{\cB(I_0)}(\theta\lambda, \mu), \\
\Hom_{\cB(I_0)}(\theta\lambda, \mu)\ni r &\longmapsto \iota(r)v\in \Hom_{\cB(I_0), \cA(I_0)}(\iota\lambda, \iota\mu)
\end{align}
\end{subequations}
are mutually inverse of each other. By (i) and (ii), we get the statement.
\end{proof}

As a consequence of the above theorem, we get the following.
\begin{corollary}
Two inductions $\alpha^{g; +}$ and $\alpha^-$ define functors from $\Delta_{\cB}^{g'}(I_0)$ (considered as a full and replete subcategories of $\End(\cB(I_0))$) to $\End(\cA(I_0))$ which act as the inclusion map on arrows.
\end{corollary}

\begin{proof}
The statement follows from the inclusion $\Hom_{\cB(I_0)}(\lambda, \mu)\hookrightarrow \Hom_{\cB(I_0), \cA(I_0)}(\iota\lambda, \iota\mu)$.
\end{proof}

We now consider the global intertwiners between induced endomorphisms. Since all intertwiners which appeared in the proof of Theorem~\ref{thm_galpha_intertwiner} are global intertwiners by Lemma~\ref{lemma_intertwiners_GLoc}, one can prove the global analogue of Theorem~\ref{thm_galpha_intertwiner}. We note that maps with exactly the same expression as Eq.~\eqref{eq:frobenius_reciprocity_for_main_formula_galpha} gives a bijection between $\Hom_{\cB_\infty, \cA_\infty}(\iota\lambda, \iota\mu)$ and $\Hom_{\cB_\infty}(\theta\lambda, \mu)$. Since we have $\Hom_{\cB(I_0)}(\theta\lambda, \mu) = \Hom_{\cB_\infty}(\theta\lambda, \mu)$ by Lemma~\ref{lemma_intertwiners_GLoc}, this observation leads to the following.

\begin{proposition}
Let $g\in G$ and $\lambda\in \Delta_{\cB}^{g'}(I_0)$. Then the space of local intertwiners between $\gpalpha{g}{\lambda}$ and $\gpalpha{g}{\mu}$ (resp. $\gmalpha{\lambda}$ and $\gmalpha{\mu}$) coincides with the space of global intertwiners between them.
\end{proposition}

\subsection{Subsectors of $[\gpalpha{g}{\lambda}]$ and $[\gmalpha{\lambda}]$} \label{subsection:ind_for_twisted_reps_common_subsector}

Let $g\in G$. We consider endomorphisms of $\cA_\infty$ which are subobjects of both $\gpalpha{g}{\lambda}$ and $\gmalpha{\mu}$ for some $\lambda, \mu \in \Delta_{\cB}^{g'}(I_0)$.
The goal of this subsection is to show the following theorem, which says such endomorphisms are $g$-localized in $I_0$ and transportable.

\begin{theorem} \label{thm:common_subsector_of_plusminus}
Let $g\in G$ and $\beta\in \End(\cA_\infty)$. Suppose that $\beta$ is a subobject of both $\gpalpha{g}{\lambda}$ and $\gmalpha{\mu}$ for some $\lambda, \mu \in \Delta_{\cB}^{g'}(I_0)$ by isometries in $\cA(I_0)$. Then we have $\beta\in \Delta_{\cA}^g(I_0)$.
\end{theorem}

Before the proof of above theorem, we give the following two lemmas.

\begin{lemma} 
Let $\lambda\in \Delta_{\cB}^{g'}(I_0)$ and $\beta\in \End(\cA_\infty)$ such that $\beta$ is a subobject of $\gmalpha{\lambda}$ in $\End(\cA_\infty)$. Suppose that $\beta$ is a subobject of $\gpalpha{g}{\mu}$ in $\End(\cA_\infty)$ for some $\mu\in \Delta_{\cB}^{g'}(I_0)$. Then $\beta$ is a subobject of $\gpalpha{g}{\lambda}$ in $\End(\cA_\infty)$. An Analogous statement holds if one interchange the plus and minus.
\end{lemma}

\begin{lemma} \label{lemma:galpha_sub_intertwiners}
Let $\beta_i\in \End(\cA_\infty)$ such that $\beta_i$ is a subobject of $\gpalpha{g}{\lambda_i}$ for $\lambda_i\in \Delta_{\cB}^{g'}(I_0)$, $i=1, 2$. Then we have 
\begin{equation*}
\Hom_{\cB_\infty, \cA_\infty}(\beta_1\iota, \beta_2\iota) = \Hom_{\cA_\infty}(\beta_1, \beta_2),
\end{equation*}
where $\Hom_{\cB_\infty, \cA_\infty}(\beta_1\iota, \beta_2\iota)$ denotes the space of global intertwiners $\{t\in \cA_\infty : t\beta_1\iota(n) = \beta_2\iota(n)t \ \text{for all} \ n\in \cB_\infty \}$.
An analogous statement holds if one interchange the plus and minus.
\end{lemma}

The proofs of above two lemmas are given by essentially the same arguments as in the proofs of Lemma 3.1 and 3.2 of \cite{BE3}. 

\begin{proof}[Proof of Theorem~\ref{thm:common_subsector_of_plusminus}]
By our assumption, there exists an isometry $t\in \Hom_{\cA_\infty}(\beta, \gpalpha{g}{\lambda})$ such that $t\in \cA(I_0)$. We set $\beta^+(\cdot)\equiv  \beta(\cdot) = t^*\gpalpha{g}{\lambda}(\cdot)t$ and $\beta^-(\cdot)=t^*\gmalpha{\lambda}(\cdot)t$. 

We first show $\beta^+ = \beta^-$. For every $n\in \cB_\infty$, we have
\begin{equation*}
\beta^+(n) = t^*\gpalpha{g}{\lambda}(n) t = t^* \lambda(n) t = t^*\gmalpha{\lambda}(n) t = \beta^-(n),
\end{equation*}
which implies $1\in \Hom_{\cB_\infty, \cA_\infty}(\beta^+\iota, \beta^-\iota)$. 
Since $\beta^+\prec \gmalpha{\mu}$ and $\beta^-\prec \gmalpha{\lambda}$, we have $1\in \Hom_{\cA_\infty}(\beta^+, \beta^-)$ by Lemma~\ref{lemma:galpha_sub_intertwiners} and obtain $\beta = \beta^+ = \beta^-$. 

We then show that $\beta$ is $g$-localized in $I_0$. For $I_+ \in \cJ$ with $I_0 < I_+$ and $m\in \cA(I_+)$, we have $\gpalpha{g}{\lambda}(m) = \tilbeta_g(m)$ by (i) of Proposition~\ref{prop:galpha_localizable}. Since $t$ and $\tilbeta_g(m)$ commute, we get $\beta^+(m) = t^*\gpalpha{g}{\lambda}(m)t = t^*\tilbeta_g(m)t = \tilbeta_g(m)$. Similarly, for $I_- \in \cJ$ with $I_- < I_0$ and $m\in \cA(I_-)$, we have $\gmalpha{\lambda}(m) = m$ by (ii) of Proposition~\ref{prop:galpha_localizable} and we obtain $\beta^+(m) = \beta^-(m) = t^*\gmalpha{\lambda}(m)t = t^*mt = m$. Therefore we see that $\beta$ is $g$-localized in $I_0$.

Finally, by Lemma~\ref{lemma:galpha_transportability} one can also see that $\beta$ is transportable, which completes the proof.
\end{proof}

\subsection{$\alpha\sigma$-reciprocity formula} \label{subsection:ind_for_twisted_reps_alpha_sigma_reciprocity}

We now consider a relation to the restriction procedure. We will generalize the $\alpha\sigma$-reciprocity formula of~\cite{BE1}. The arguments below are essentially the same as the case of usual $\alpha$-induction, although some computations become complicated.

We first recall the definition of $\sigma$-restriction~\cite{LR95,BE1}.

\begin{definition}
For $\beta\in\End(\cA_\infty)$, the \emph{$\sigma$-restricted} endomorphism $\sigma_\beta\in\End(\cB_\infty)$ is defined by
\[
\sigma_\beta = \iotabar\circ\beta\circ\iota.
\]
\end{definition}

Let $g\in G$ and $\beta\in \Delta_{\cA}^g(I_0)$. It is easy to see that $\sigma_\beta$ is $g'$-localized in $I_0$. One can also prove $\sigma_\beta\in \Delta_{\cB}^{g'}(I_0)$ as below. Let $I_1\in \cJ$ be an interval. Since $\theta$ and $\beta$ is transportable, we can choose unitaries $u_{\theta; I_0, I_1}\in \cB_\infty$ and $Q_{\beta; I, I_1}\in \cA_\infty$ such that $\theta_{I_1} = \Ad(u_{\theta; I_0, I_1})\circ \theta$ is localized in $I_1$ and $\beta_{I_1} = \Ad(Q_{\beta; I, I_1})\circ \beta$ is $g$-localized in $I_1$.

\begin{proposition} \label{prop:sigma_restriction_transportable}
Under the assumptions as above, we have that $\sigma_{\beta, I_1} = \Ad(u_{\sigma_\beta; I_0, I_1})\circ \sigma_\beta$ is $g'$-localized in $I_1$ where 
\begin{equation*}
u_{\sigma_\beta; I_0, I_1} = u_{\theta; I_0, I_1} \iotabar(Q_{\beta; I_0, I_1}).
\end{equation*}
Consequently, $\sigma_\beta$ is transportable.
\end{proposition}

\begin{proof}
Let $\widetilde{I}\in\cJ$ be an interval such that $\widetilde{I}\subset I_1^\perp$. We note that $\theta_{I_1}$ acts as identity on $\cB(\widetilde{I})$ since $\theta_{I_1}$ is localized in $I_1$. Let $n\in \cB(I)$.
To see that $\sigma_{\beta, I_1}$ is $g'$-localized in $I_1$, we compute
\begin{align*}
\sigma_{\beta, I_1}(n) &= \Ad(u_{\sigma_\beta; I_0, I_1})\circ \sigma_\beta(n) \\
&= \Ad(u_{\theta; I_0, I_1}\cdot \iotabar(Q_{\beta; I, I_1}))\circ \iotabar\circ\beta\circ\iota(n) \\
&= \Ad(u_{\theta; I_0, I_1})\circ \iotabar\circ \beta_{I_1}\circ\iota(n).
\end{align*}
If $\widetilde{I} < I_1$, then $\beta_{I_1}\circ\iota(n) = \iota(n)$ and we get
\begin{equation*}
\Ad(u_{\sigma_\beta; I_0, I_1})\circ \iotabar\circ \beta_{I_1}\circ\iota(n) = \Ad(u_{\theta; I_0, I_1})\circ \theta(n) = \theta_1(n) = n.
\end{equation*}
Also, if $\widetilde{I} > I_1$, then $\beta_{I_1}\circ \iota(n) = \tilbeta_g\circ\iota(n) = \iota\circ\beta_{g'}(n)$ and we get
\begin{align*}
\Ad(u_{\theta; I_0, I_1})\circ \iotabar\circ \beta_{I_1}\circ\iota(n) &= \Ad(u_{\theta; I_0, I_1})\circ \iotabar\circ \iota\circ\beta_{g'}(n) \\
&= \Ad(u_{\theta; I_0, I_1})\circ \theta\circ\beta_{g'}(n) \\
&= \theta_{I_1}\circ \beta_{g'}(n) = \beta_{g'}(n).
\end{align*}
This completes the proof.
\end{proof}

We fix intervals $I_-, I_+\in \cJ$ such that $I_- < I_0 < I_+$ and choose unitaries $u_{\theta, \pm} = u_{\theta, I_0, I_\pm}$ and $Q_{\beta, \pm} = Q_{\beta;I_0, I_\pm}$ as above. We set $u_{\sigma_\beta, \pm} = u_{\sigma_\beta;I_0, I_\pm}$ as in the above Proposition.

\begin{lemma} \label{lemma:sigma_restriction_and_theta_braiding_formula}
Under the assumptions as above, we have
\begin{enumerate}
\item $c^+(\sigma_\beta, \theta) = \gamma_{g'}(\theta)(\iotabar(Q_{\beta, -}^*))\cdot c^+(\theta, \gamma_{g'}(\theta))\cdot \iotabar(Q_{\beta, -})$.
\item $c^-(\sigma_\beta, \theta) = \theta(\iotabar(Q_{\beta, +})^*)\cdot c^-(\theta, \theta)\cdot \iotabar(Q_{\beta, +})$.
\end{enumerate}
\end{lemma}

\begin{proof}
(i) By the definition of the $G$-crossed braiding, we have $c^+(\sigma_\beta, \theta) = \gamma_{g'}(\theta)(u_{\sigma_\beta, -}^*) u_{\sigma_\beta, -}$. We also have $c^+(\theta, \gamma_{g'}(\theta)) = \gamma_{g'}(\theta)(u_{\theta, -}^*)u_{\theta, -}$. Using these formulas, we compute
\begin{align*}
c^+(\sigma_\beta, \theta) &= \gamma_{g'}(\theta)(u_{\sigma_\beta, -}^*)u_{\sigma_\beta, -} \\
&=\gamma_{g'}(\theta)(\iotabar(Q_{\beta, -}^*)u_{\theta, -}^*) u_{\theta, -}\iotabar(Q_{\beta, -}) \\
&= \gamma_{g'}(\theta)(\iotabar(Q_{\beta, -}^*))\cdot c^+(\theta, \gamma_{g'}(\theta))\cdot\iotabar(Q_{\beta, -}).
\end{align*}

\smallskip

(ii) Using Lemma~\ref{lemma:GLoc_braiding_another_formula}, we have $c^-(\sigma_\beta, \theta) = \theta(u_{\sigma_\beta, +}^*)u_{\sigma_\beta, +}$ and $c^-(\theta, \theta) = \theta(u_{\theta, +}^*)u_{\theta, +}$. Using these formulas, one can similarly get the equality.
\end{proof}

For $g\in G$, $\lambda \in \Delta_{\cB}^{g'}(I_0)$ and $\beta \in \Delta_{\cA}^g(I_0)$, we denote by $\Hom_{\cB(I_0), \cA(I_0)}(\iota\lambda, \beta\iota)$ the space of local intertwiners $\{t\in \cA(I_0) : t\iota\lambda(n) = \beta\iotabar(n)t \ \text{for all} \ n\in \cB(I_0)\}$.

\begin{theorem} \label{thm:alpha_sigma_reciprocity}
Let $g\in G$, $\lambda \in \Delta_{\cB}^{g'}(I_0)$ and $\beta \in \Delta_{\cA}^g(I_0)$. Then we have the following.
\begin{enumerate}
\item $\Hom_{\cA(I_0)}(\gpalpha{g}{\lambda}, \beta) = \Hom_{\cB(I_0), \cA(I_0)}(\iota\lambda, \beta\iota)$.
\item $\Hom_{\cA(I_0)}(\gmalpha{\lambda}, \beta) = \Hom_{\cB(I_0), \cA(I_0)}(\iota\lambda, \beta\iota)$.
\end{enumerate}
In particular, we have the following $\alpha\sigma$-reciprocity formula:
\begin{equation*}
\langle\gpalpha{g}{\lambda}, \beta \rangle_{\cA(I_0)} = \langle \lambda, \sigma_\beta \rangle_{\cB(I_0)} = \langle \gmalpha{\lambda}, \beta \rangle_{\cA(I_0)}.
\end{equation*}
\end{theorem}

\begin{proof}
(i) Since $\gpalpha{g}{\lambda}$ acts as $\lambda$ on $\cB(I_0)$, we get $\Hom_{\cA(I_0)}(\gpalpha{g}{\lambda}, \beta) \subset \Hom_{\cB(I_0), \cA(I_0)}(\iota\lambda, \beta\iota)$. 
Let $t\in \Hom_{\cB(I_0), \cA(I_0)}(\iota\lambda, \beta\iota)$. We set $s=\iotabar(t)\in \Hom_{\cB(I_0)}(\theta\lambda, \sigma_\beta)$.
To see $t\in \Hom_{\cA(I_0)}(\gpalpha{g}{\lambda}, \beta)$, it suffices to show \[
t\gpalpha{g}{\lambda}(v) = \beta(v)t.
\]
Applying $\iotabar$ to the both sides and by Eq.~\eqref{eq:values_of_galpha_on_v}, it is equivalent to the equality
\[
s\theta(c^+(\lambda, \theta)^* z_g^*)\iotabar(v) = \iotabar\beta(v)s.
\]
By the properties of braiding Eq.~\eqref{naturality_gcbraiding} and Eq.~\eqref{braid_rel_gcbraiding} and the intertwining property of $z_g$, we have
\begin{equation*}
s\theta(c^+(\lambda, \theta)^*z_g^*) = c^+(\sigma_\beta, \theta)^* \cdot \gamma_{g'}(\theta)(s)\cdot z_g^* c^+(\theta, \theta).
\end{equation*}
We now compute
\begin{align*}
s\theta(c^+(\lambda, \theta)^* z_g^*)\iotabar(v) &= c^+(\sigma_\beta, \theta)^* \cdot \gamma_{g'}(\theta)(s) \cdot z_g^* c^+(\theta, \theta)\iotabar(v) \\
&= c^+(\sigma_\beta, \theta)^* \cdot \gamma_{g'}(\theta)(s) \cdot z_g^*\iotabar(v) \\
&= c^+(\sigma_\beta, \theta)^* z_g^* \theta(s)\iotabar(v) \\
&= c^+(\sigma_\beta, \theta)^* z_g^* \iotabar\iota\iotabar(t)\iotabar(v) \\
&= c^+(\sigma_\beta, \theta)^* z_g^* \iotabar(v) \iotabar(t) \\
&= c^+(\sigma_\beta, \theta)^* z_g^* \iotabar(v)s,
\end{align*}
where we used Lemma~\ref{lemma:locality_of_Q_system_net_of_subfactor} in the first line.
Using (i) of the previous lemma and Lemma~\ref{lemma:locality_of_Q_system_net_of_subfactor} again, we continue
\begin{align*}
c^+(\sigma_\beta, \theta)^* z_g^* \iotabar(v) s &=
\iotabar(Q_{\beta, -}^*)c^+(\theta, \gamma_{g'}(\theta))^* \cdot \gamma_{g'}(\theta)(\iotabar(Q_{\beta, -}))\cdot z_g^* \iotabar(v) s \\
&= \iotabar(Q_{\beta, -}^*)c^+(\theta, \gamma_{g'}(\theta))^* z_g^* \theta(\iotabar(Q_{\beta, -})) \iotabar(v) s \\
&= \iotabar(Q_{\beta, -}^*)c^+(\theta, \gamma_{g'}(\theta))^* z_g^* \iotabar(v) \iotabar(Q_{\beta, -}) s \\
&= \iotabar(Q_{\beta, -}^*)\theta(z_g^*) c^+(\theta, \theta)^* \iotabar(v) \iotabar(Q_{\beta, -}) s \\
&= \iotabar(Q_{\beta, -}^*)\theta(z_g^*) \iotabar(v) \iotabar(Q_{\beta, -}) s = \iotabar(Q_{\beta, -}^* z_g^* v Q_{\beta, -}) s.
\end{align*}
Since $z_g^*v = \tilbeta_g(v) = \beta_{I_-}(v) \equiv \Ad(Q_{\beta, I_-})\circ \beta (v)$, the last line is
\begin{equation*}
\iotabar(Q_{\beta, -}^* z_g^* v Q_{\beta, -}) s = \iotabar\beta(v) s.
\end{equation*}
Thus we get $s\theta(c^+(\lambda, \theta)z_g^*)\iotabar(v) = \iotabar\beta(v)s$, which is the desired equality. Therefore we showed the equality $\Hom_{\cA(I_0)}(\gpalpha{g}{\lambda}, \beta) = \Hom_{\cB(I_0), \cA(I_0)}(\iota\lambda, \beta\iota)$.

\smallskip

(ii) The statement is checked similarly by using (ii) of the previous lemma.

\smallskip

The last formulas follow from the Frobenius reciprocity, namely there is a linear bijection between $\Hom_{\cB(I_0), \cA(I_0)}(\iota\lambda, \beta\iota)$ and $\Hom_{\cB(I_0)}(\lambda, \sigma_\beta)$ by the Frobenius reciprocity. Explicitly, two linear maps given by
\begin{subequations} \label{eq:alpha_sigma_reciprocity_maps}
\begin{align}
\Hom_{\cB(I_0), \cA(I_0)}(\iota\lambda, \beta\iota)&\ni t \mapsto \iotabar(t)w \in \Hom_{\cB(I_0)}(\lambda, \sigma_\beta), \\
\Hom_{\cB(I_0)}(\lambda, \sigma_\beta)&\ni r \mapsto v^*\iota(r) \in \Hom_{\cB(I_0), \cA(I_0)}(\iota\lambda, \beta\iota), \label{eq:alpha_sigma_reciprocity_map2}
\end{align}
\end{subequations}
are mutually inverses of each other. By (i) and (ii), we get the statement.
\end{proof}

From the proof of last statement, we have the following corollary.

\begin{corollary} \label{cor:beta_contained_in_galpha_sigma_beta}
Let $g\in G$ and $\beta \in \Delta_{\cA}^g(I_0)$. Then we have $v\in \Hom_{\cA(I_0)}(\sigma_\beta, \gpalpha{g}{\sigma_\beta})$ and $v\in \Hom_{\cA(I_0)}(\sigma_\beta, \gmalpha{\sigma_\beta})$. As a consequence, $\sigma_\beta$ is a subobject of both $\gpalpha{g}{\sigma_\beta}$ and $\gmalpha{\sigma_\beta}$.
\end{corollary}

\begin{proof}
Since $\sigma_\beta\in \Delta_{\cB}^{g'}(I_0)$ by Lemma~\ref{prop:sigma_restriction_transportable} and $1\in \Hom_{\cB(I_0)}(\sigma_\beta, \sigma_\beta)$, we have 
\[
v^* \in \Hom_{\cB(I_0, \cA(I_0))}(\iota\sigma_\beta, \beta\iota).
\]
Thus we have $v\in \Hom_{\cA(I_0)}(\beta, \gpalpha{g}{\sigma_\beta})$ and $v\in \Hom_{\cA(I_0)}(\sigma_\beta, \gmalpha{\sigma_\beta})$. The rest of the statement follows from the fact that $\frac{1}{\sqrt{d\iota}}\cdot v$ is an isometry.
\end{proof}

We stated Theorem~\ref{thm:alpha_sigma_reciprocity} and Corollary~\ref{cor:beta_contained_in_galpha_sigma_beta} in terms of local intertwiner spaces. As in Section~\ref{subsection:ind_for_twisted_reps_intertwiners}, one can prove the global analogue of them.

\subsection{Relative braiding} \label{subsection:ind_for_twisted_reps_braiding}

In this final subsection, we will see that the braiding of $\GLoc \cA$ can be given by the braiding of $\GpLoc \cB$. We first see the relation between the braiding on $\GpLoc \cB$ and induced endomorphisms.

\begin{lemma} \label{lemma:basic_lemma_for_relative_braiding}
Let $g, h\in G$. For every $\rho\in \Delta_{\cB}^{g'}(I_0)$, $\lambda, \mu \in \Delta_{\cB}^{h'}(I_0)$, and $r\in \Hom_{\cB(I_0), \cA(I_0)}(\iota\lambda, \iota\mu)$, we have the following:
\begin{enumerate}
\item $\tilbeta_g(r) c^+(\rho, \lambda) = c^+(\rho, \mu)\gpalpha{g}{\rho}(r)$.
\item $rc^-(\rho, \lambda) = c^-(\rho, \mu)\gmalpha{\rho}(r)$.
\item $\tilgamma_h(\gmalpha{\rho})(r)c^+(\lambda, \rho) = c^+(\mu, \rho)r$.
\end{enumerate}
\end{lemma}

\begin{proof}
(i) We set $s = \iotabar(r)\in \Hom(\theta\lambda, \theta\mu)$. Using the properties of braiding Eq.~\eqref{naturality_gcbraiding} and Eq.~\eqref{braid_rel_gcbraiding}, we compute $c(\rho, \theta\mu)\rho(s)$ in two ways:
\begin{align*}
c^+(\rho, \theta\mu)\rho(s) &= \gamma_{g'}(s)c^+(\rho, \theta\lambda) = \beta_{g'}(s)c^+(\rho, \theta\lambda) \\
&= \beta_{g'}(s)\cdot \gamma_{g'}(\theta)(c^+(\rho, \lambda))\cdot c^+(\rho, \theta),
\end{align*}
and
\begin{equation*}
c^+(\rho, \theta\mu)\rho(s) = \gamma_{g'}(\theta)(c^+(\rho, \theta))\cdot c^+(\rho, \theta)\rho(s).
\end{equation*}
Thus we get
\begin{equation*}
\gamma_{g'}(\theta)(c^+(\rho, \theta))\cdot c^+(\rho, \theta)\rho(s) c^+(\rho, \theta)^* = \beta_{g'}(s)\cdot \gamma_{g'}(\theta)(c^+(\rho, \lambda)).
\end{equation*}
Applying $\Ad(z_g)$ to the both sides of the above equality, the left-hand side is
\begin{align*}
z_g\cdot \gamma_{g'}(\theta)(c^+(\rho, \theta))\cdot c^+(\rho, \theta)\rho(s) c^+(\rho, \theta)^* z_g^* &= \theta(c^+(\rho, \theta)) z_g c^+(\rho, \theta)\rho(s) c^+(\rho, \theta)^* z_g^* \\
&= \theta(c^+(\rho, \theta)) \cdot \Ad(z_g c^+(\rho, \theta))\circ \rho(s),
\end{align*}
and the right-hand side is
\begin{align*}
z_g\beta_{g'}(s)\cdot \gamma_{g'}(\theta)(c^+(\rho, \lambda))\cdot  z_g^* &= z_g \beta_{g'}\iotabar(r) \cdot \gamma_{g'}(\theta)(c^+(\rho, \lambda)) \cdot z_g^* \\
&= \iotabar\tilbeta_g(r)z_g \cdot \gamma_{g'}(\theta)(c^+(\rho, \lambda))\cdot z_g^* \\
&= \iotabar\tilbeta_g(r) z_g z_g^* \theta(c^+(\rho, \lambda)) \\
&= \iotabar\tilbeta_g(r) \theta(c^+(\rho, \lambda)),
\end{align*}
where we used the intertwining property of $z_g$ Eq.~\eqref{eq:z_g_intertwining_property}.
Hence we have
\begin{equation*}
\theta(c^+(\rho, \theta)) \Ad(z_g c^+(\rho, \theta))\circ \rho(s) = \iotabar\tilbeta_g(r) \theta(c^+(\rho, \lambda)).
\end{equation*}
We now apply $\iotabar^{-1}$ to both sides and get the statement.
\smallskip

(ii) The statement follows by using a similar argument to the one in (i). Set $s = \iotabar(r)\in \Hom(\theta\lambda, \theta\mu)$. Then we compute $c^-(\rho, \theta\mu)\rho(s)$ in two ways as before:
\begin{equation*}
c^-(\rho, \theta\mu)\rho(s) = sc^-(\rho, \theta\lambda)= s\theta(c^-(\rho, \lambda))c^-(\rho, \theta),
\end{equation*}
and 
\begin{equation*}
c^-(\rho, \theta\mu)\rho(s) = \theta(c^-(\rho, \mu))c^-(\rho, \theta)\rho(s).
\end{equation*}
Thus we get
\begin{equation*}
\theta(c^-(\rho, \mu))c^-(\rho, \theta)\rho(s) = s\theta(c^-(\rho, \lambda))c^-(\rho, \theta).
\end{equation*}
Applying $\iotabar^{-1}$, we get the statement.

\smallskip

(iii) Using (ii), we have $r^* c^-(\gamma_h(\rho), \mu) = c^-(\gamma_h(\rho), \lambda)\gmalpha{\gamma_h(\rho)}(r^*)$. Taking $*$ and by the covariance property Proposition~\ref{prop:covariance_galpha}, we get
$c^-(\gamma_{h}(\rho), \mu)^* r = \tilgamma_{h}(\gmalpha{\rho})(r) c^-(\gamma_h(\rho), \lambda)^*$. Since $c^-(\gamma_h(\rho), \lambda)^* = c^+(\rho, \lambda)$ and $c^-(\gamma_h(\rho), \mu)^* = c^+(\rho, \mu)$ by the definition of $c^-$, we get the statement.
\end{proof}

Using the above lemma, we have the following commutative relation between induced endomorphisms.

\begin{proposition}
Let $g, h\in G$, $\lambda\in \Delta_{\cB}^{g'}(I_0)$, and $\mu\in \Delta_{\cB}^{h'}(I_0)$. Then we have
\begin{equation*}
\Ad(c^+(\lambda, \mu))\circ \gpalpha{g}{\lambda} \circ \gmalpha{\mu} = \tilgamma_g(\gmalpha{\mu})\circ \gpalpha{g}{\lambda}.
\end{equation*}
\end{proposition}

\begin{proof}
Since $\gpalpha{g}{\lambda}$ and $\gmalpha{\mu}$ act as $\lambda$ and $\mu$ on $\cB$ respectively, the equality holds on $\cB$. Thus it suffices to show that 
\begin{equation*}
c^+(\lambda, \mu)\cdot \gpalpha{g}{\lambda} \circ \gmalpha{\mu}(v) = \tilgamma_g(\gmalpha{\mu})\circ \gpalpha{g}{\lambda}(v) \cdot c^+(\lambda, \mu)^.
\end{equation*}
We recall that we have $\gpalpha{g}{\lambda}(v) = c^+(\lambda, \theta) z_g^* v$ and $\gmalpha{\mu}(v) = c^-(\mu, \theta)^* v = c^+(\theta, \mu) v$. Using Proposition~\ref{prop:covariance_galpha}, we also have $\tilgamma_g(\gmalpha{\mu}(v)) = \gmalpha{\gamma_{g'}(\mu)}(v) = c^-(\gamma_{g'}(\mu), \theta)^* v = c^+(\theta, \gamma_{g'}(\mu)) v$. Then the left-hand side is
\begin{align*}
c^+(\lambda, \mu)\cdot \gpalpha{g}{\lambda} \circ \gmalpha{\mu}(v) &= c^+(\lambda, \mu)\gpalpha{g}{\lambda}(c^+(\theta, \mu) v) \\
&= c^+(\lambda, \mu)\lambda(c^+(\theta, \mu))c^+(\lambda, \theta)^* z_g^* v,
\end{align*}
and the right-hand side is
\begin{align*}
\tilgamma_g(\gmalpha{\mu})\circ \gpalpha{g}{\lambda}(v) \cdot c^+(\lambda, \mu) &= \tilgamma_g(\gmalpha{\mu})(c^+(\lambda, \theta)^* z_g^* v)\cdot c^+(\lambda, \mu) \\
&= \gamma_{g'}(\mu)(c^+(\lambda, \theta)^* z_g^*)\cdot c^+(\theta, \gamma_{g'}(\mu))v c^+(\lambda, \mu) \\
&= \gamma_{g'}(\mu)(c^+(\lambda, \theta)^* z_g^*)\cdot c^+(\theta, \gamma_{g'}(\mu)) \theta(c^+(\lambda, \mu)) v.
\end{align*}
Hence it is enough to check the equality
\begin{equation*}
c^+(\lambda, \mu)\lambda(c^+(\theta, \mu))c^+(\lambda, \theta)^* z_g^* = \gamma_{g'}(\mu)(c^+(\lambda, \theta)^* z_g^*)\cdot c^+(\theta, \gamma_{g'}(\mu)) \theta(c^+(\lambda, \mu)),
\end{equation*}
or equivalently
\begin{equation*}
\gamma_{g'}(\mu)(z_g c^+(\lambda, \theta))\cdot c^+(\lambda, \mu)\lambda(c^+(\theta, \mu)) = c^+(\theta, \gamma_{g'}(\mu)) \theta(c^+(\lambda, \mu)) z_g c^+(\lambda, \theta).
\end{equation*}
To see this, we compute the right-hand side
\begin{align*}
c^+(\theta, \gamma_{g'}(\mu)) \theta(c^+(\lambda, \mu)) z_g c^+(\lambda, \theta) &= c^+(\theta, \gamma_{g'}(\mu))z_g \cdot \gamma_{g'}(\theta)(c^+(\lambda, \mu))\cdot  c^+(\lambda, \theta) \\
&= \gamma_{g'}(\mu)(z_g)\cdot c^+(\gamma_{g'}(\theta), \gamma_{g'}(\mu)) \cdot \gamma_{g'}(\theta)(c^+(\lambda, \mu)) \cdot c^+(\lambda, \theta) \\
&= \gamma_{g'}(\mu)(z_g)\cdot \gamma_{g'}(\mu)(c^+(\lambda, \theta))\cdot c^+(\lambda, \mu)\lambda(c^+(\theta, \mu)),
\end{align*}
where we used the YBE Eq.~\eqref{eq:YBE_for_G_crossed_braiding} in the last line. Hence we get the desired equality, which completes the proof.
\end{proof}

Let $g, h\in G$, $\lambda\in \Delta_{\cB}^{g'}(I_0)$, and $\mu\in \Delta_{\cB}^{h'}(I_0)$. Suppose that we have $\beta, \delta\in \End(\cA_\infty)$ such that $\beta$ is a subobject of $\gpalpha{g}{\lambda}$ and $\delta$ is a subobject of $\gmalpha{\mu}$. We denote corresponding isometries by $t\in \Hom(\beta, \gpalpha{g}{\lambda})$ and $s\in \Hom(\delta, \gmalpha{\mu})$.
Then we define the operator \[ 
c_r^+(\beta, \delta) = \tilgamma_g(s^*)\tilgamma_g(\gmalpha{\mu})(t^*) c^+(\lambda, \mu) \gpalpha{g}{\lambda}(s)t.
\]
Thanks to the above proposition, we have $c_r^+(\beta, \delta)\in \Hom(\beta\delta, \tilgamma_g(\delta)\beta)$. Using Lemma~\ref{lemma:basic_lemma_for_relative_braiding}, one can see the following.

\begin{lemma}
Under the notations as above, the following holds.
\begin{enumerate}
\item $c_r^+(\beta, \delta)$ is a unitary.
\item $c_r^+(\beta, \delta)$ does not depend on the choie of $\lambda, \mu$ and isometries $s, t$ in the following sense: If there exist isometries $\tilde{t}\in \Hom(\beta, \gpalpha{g}{\tilde{\lambda}})$ and $\tilde{s}\in \Hom(\delta, \gmalpha{\tilde{\mu}})$ for some $\tilde{\lambda}\in \Delta_{\cB}^{g'}(I_0)$ and $\tilde{\mu}\in \Delta_{\cB}^{h'}(I_0)$m, then we have
\[
c_r^+(\beta, \delta) = \tilgamma_g(\tilde{s}^*)\tilgamma_g(\gmalpha{\tilde{\mu}})(\tilde{t}^*) c^+(\tilde{\lambda}, \tilde{\mu}) \gpalpha{g}{\tilde{\lambda}}(\tilde{s})\tilde{t}.
\]
\end{enumerate}
\end{lemma}

\begin{proof}
(i) The statement can be checked by direct computation as in the proof of~\cite[Proposition 3.26]{BE1}.

\smallskip

(ii) This statement can also be checked by direct computation as in the proof of~\cite[Lemma 3.11]{BE3}.
\end{proof}

Let $g, h\in G$, $\beta\in \Delta_{\cA}^g(I_0)$ and $\delta\in \Delta_{\cA}^h(I_0)$. In the rest of this subsection, we will see that $c_r^+(\beta, \delta)$ coincides with the $G$-crossed braiding $c^+(\beta, \delta)$. We recall that $\beta$ is a subobject of $\gpalpha{g}{\sigma_\beta}$ and $\delta$ is a subobject of $\gmalpha{\sigma_\beta}$ by Corollary~\ref{cor:beta_contained_in_galpha_sigma_beta}. By the same corollary, $c_r^+(\beta, \delta)$ is given by
\[
c_r^+(\beta, \delta) = \frac{1}{(d\iota)^2}\cdot \tilgamma_g(v^*)\tilgamma_g(\gmalpha{\sigma_\delta})(v^*) c^+(\sigma_\beta, \sigma_\delta) \gpalpha{g}{\sigma_\beta}(v) v.
\]
To prove $c_r^+(\beta, \delta) = c^+(\beta, \delta)$, we needs some preparations.
Let us fix intervals $I_\pm \in \cJ$ such that $I_- < I_0 < I_+$. As in Section~\ref{subsection:ind_for_twisted_reps_alpha_sigma_reciprocity}, we choose unitaries $u_{\theta, \pm}\in \cB_\infty$ and $Q_{\beta, -}, Q_{\delta, +}\in \cA_\infty$ such that $\theta_{I_\pm} = \Ad(u_{\theta, \pm})\circ \theta\in \Delta_{\cB}^e(I_\pm)$, $\beta_{I_-} = \Ad(Q_{\beta, -})\circ \beta\in \Delta_{\cA}^g{I_-}$ and $\delta_{I_+} = \Ad(Q_{\delta, +})\circ \delta\in \Delta_{\cA}^h(I_+)$.
Then we have
\begin{equation*}
c^+(\sigma_\beta, \theta) = \gamma_{g'}(\theta)(\iotabar(Q_{\beta, -}^*))\cdot c^+(\theta, \gamma_{g'}(\theta))\iotabar(Q_{\beta, -}).
\end{equation*}
by (i) of Lemma~\ref{lemma:galpha_transportability}. For the later use, we derive some formulas as below.
By Proposition~\ref{prop:sigma_restriction_transportable}, we have $\sigma_{\delta, +} = \Ad(u_{\sigma_\delta, +})\circ \sigma_\delta\in \Delta_{\cA}^h(I_+)$ with $u_{\sigma_\delta, +} = u_{\theta, +}\iotabar(Q_{\delta, +})$. Using Lemma~\ref{lemma:GLoc_braiding_another_formula}, we find
\begin{align*}
c^+(\sigma_\beta, \sigma_\delta) &= \gamma_{g'}(u_{\sigma_\delta, +}^*)\sigma_\beta(u_{\sigma_\delta, +})\\
&= \beta_{g'}(\iotabar(Q_{\delta, +}^*) u_{\theta, +}^*) \sigma_\beta(u_{\theta, +}\iotabar(Q_{\delta, +})) \\
&= \beta_{g'}(\iotabar(Q_{\delta, +}^*)) c^+(\sigma_\beta, \theta) \sigma_\beta\iotabar(Q_{\delta, +}).
\end{align*}
Moreover, we see
\begin{align*}
\gpalpha{g}{\sigma_\beta}(v)v &= v\beta(v) = v Q_{\beta, -}^* \beta_{I_-}(v) Q_{\beta, -}^* \\
&= vQ_{\beta, -}^*\tilbeta_g(v)Q_{\beta, -} = \iota\iotabar(Q_{\beta, -}^*)v\tilbeta_g(v)Q_{\beta, -} \\
&= \iota\iotabar(Q_{\beta, -}^*)v z_g^*v Q_{\beta, -} = \iota\iotabar(Q_{\beta, -}^*)\theta(z_g^*)vv Q_{\beta, -}.
\end{align*}
and
\begin{align*}
\tilgamma_g(\gmalpha{\sigma_\delta})(v)\tilgamma_g(v) &= \tilgamma_g(v)\tilgamma_g(\delta)(v) = \tilbeta_g(v)\tilgamma_g(\delta)(v) \\
&= z_g^* v \tilbeta_g \circ \delta\circ \tilbeta_g^{-1}(v) = z_g^* v \tilbeta_g(Q_{\delta, +}^* \delta_{I_+}\circ \tilbeta_g^{-1}(v) Q_{\delta, +}) \\
&= z_g^* v \tilbeta_g(Q_{\delta, +}^* \tilbeta_g^{-1}(v) Q_{\delta, +}) = z_g^* \iotabar(\tilbeta_g(Q_{\delta, +}^*)) vv \tilbeta_g(Q_{\delta, +}) \\
&= \beta_{g'}\iotabar(Q_{\delta, +}^*) z_g^* vv \tilbeta_g(Q_{\delta, +}).
\end{align*}
We now prove the following theorem.

\begin{theorem}
Let $g, h\in G$, $\beta\in \Delta_{\cA}^g(I_0)$, and $\delta\in \Delta_{\cA}^h(I_0)$. Then we have
\begin{equation*}
c_r^+(\beta, \delta) = c^+(\beta, \delta)
\end{equation*}
where $c^+(\beta, \delta)$ denotes the braiding on $\GLoc \cA$.
\end{theorem}

\begin{proof}
Using the formulas above, we compute
\begin{align*}
(d\iota)^2\cdot c_r^+(\beta, \delta) &= \tilgamma_g(v^*)\tilgamma_g(\gmalpha{\sigma_\delta})(v^*) c^+(\sigma_\beta, \sigma_\delta) \gpalpha{g}{\sigma_\beta}(v) v \\
&= \tilgamma_g(v^*)\tilgamma_g(\gmalpha{\sigma_\delta})(v^*)
\beta_{g'}(\iotabar(Q_{\delta, +}^*)) c^+(\sigma_\beta, \theta) \sigma_\beta\iotabar(Q_{\delta, +}) \gpalpha{g}{\sigma_\beta}(v)v \\
&= \tilbeta_g(Q_{\delta, +}^*) v^* v^* z_g c^+(\sigma_\beta, \theta) \sigma_\beta\iotabar(Q_{\delta, +}) \gpalpha{g}{\sigma_\beta}(v)v \\
&= \tilbeta_g(Q_{\delta, +}^*) v^* v^* z_g c^+(\sigma_\beta, \theta)\sigma_\beta\iotabar(Q_{\delta, +})v\beta(v) \\
&= \tilbeta_g(Q_{\delta, +}^*) v^* v^* z_g c^+(\sigma_\beta, \theta)\iota\iotabar\beta\iota\iotabar(Q_{\delta, +})v\beta(v) \\
&= \tilbeta_g(Q_{\delta, +}^*) v^* v^* z_g c^+(\sigma_\beta, \theta)v\beta(v)\beta(Q_{\delta, +}).
\end{align*}

We will show that the middle part $v^* v^* z_g c^+(\sigma_\beta, \theta)v\beta(v)$ is equal to $(d\iota)^2\cdot 1$.
If this is the case, then we have $c^+_r(\beta, \delta) = \tilbeta_g(Q_{\delta, +}^*)\beta(Q_{\delta, +}) = \tilgamma_g(Q_{\delta, +}^*)\beta(Q_{\delta, +}) = c^+(\beta, \delta)$ by Lemma~\ref{lemma:GLoc_braiding_another_formula}, which is the desired statement.
Let us compute
\begin{align*}
v^* v^* z_g c^+(\sigma_\beta, \theta)v\beta(v) &= v^* v^* z_g \cdot \gamma_{g'}(\theta)(\iotabar(Q_{\beta, -}^*)) \cdot c^+(\theta, \gamma_{g'}(\theta))\iotabar(Q_{\beta, -}) v\beta(v) \\
&= v^* v^* z_g \cdot \gamma_{g'}(\theta)(\iotabar(Q_{\beta, -}^*))\cdot  c^+(\theta, \gamma_{g'}(\theta)) \theta(z_g^*)vvQ_{\beta, -} \\
&= v^* v^* z_g \cdot \gamma_{g'}(\theta)(\iotabar(Q_{\beta, -}^*))\cdot z_g^* c^+(\theta, \theta)vvQ_{\beta, -} \\
&= v^* v^* \theta(\iotabar(Q_{\beta, -}^*))  c^+(\theta, \theta)vvQ_{\beta, -} \\
&= Q_{\beta, -}^* v^*v^* c^+(\theta, \theta)vvQ_{\beta, -} \\
&= Q_{\beta, -}^* v^*v^* v v Q_{\beta, -} \\
&= (d\iota)^2\cdot 1 \ .
\end{align*}
This completes the proof.
\end{proof}

\vspace{20pt}

\noindent {\bf Acknowledgments.}
The author wishes to express his gratitude to Yasuyuki Kawahigashi for his constant support and many helpful comments. The author is supported by Leading Graduate Course for Frontiers of Mathematical Sciences and Physics. He is grateful for their financial support.

\bibliographystyle{abbrv}

\end{document}